\documentclass[generic]{imsart}

\usepackage[sort&compress]{natbib}
\RequirePackage{amsthm,amsmath}
\RequirePackage[colorlinks,citecolor=blue,urlcolor=blue]{hyperref}
\RequirePackage{graphicx}

\usepackage{amssymb, mathrsfs, multirow, url, subfigure}
\usepackage{ifthen} 
\usepackage{amsfonts}
\usepackage[usenames]{color}
\usepackage{multirow}
\usepackage{multicol}
\usepackage{float}

\startlocaldefs
\theoremstyle{plain}

\newtheorem{thm}{Theorem}
\newtheorem{cor}{Corollary}

\newtheorem{lem}{Lemma}

\theoremstyle{remark}

\newtheorem{cond}{Condition}
\newtheorem*{lp}{Local Prior Condition}
\newtheorem*{gp}{Global Prior Condition}

\newcommand{\RR}{\mathbb{R}}

\newcommand{\E}{\mathsf{E}}

\newcommand{\prob}{\mathsf{P}}

\newcommand{\eps}{\varepsilon}

\newcommand{\nm}{\mathsf{N}}

\renewcommand{\S}{\mathcal{S}}
\endlocaldefs

\begin{document}

\begin{frontmatter}
\title{Empirical Bayes inference in sparse high-dimensional generalized linear models}
\runtitle{Empirical Bayes inference in sparse high-dim GLM}

\begin{aug}
\author{\fnms{Yiqi} \snm{Tang}\ead[label=e1]{ytang22@ncsu.edu}}
\and
\author{\fnms{Ryan} \snm{Martin}\ead[label=e2]{rgmarti3@ncsu.edu}}

\address{Department of Statistics,
North Carolina State University, Raleigh, US
\printead{e1,e2}}

\runauthor{Y. Tang and R. Martin}

\end{aug}

\begin{abstract}
High-dimensional linear models have been widely studied, but the developments in high-dimensional generalized linear models, or GLMs, have been slower.  In this paper, we propose an empirical or data-driven prior leading to an empirical Bayes posterior distribution which can be used for estimation of and inference on the coefficient vector in a high-dimensional GLM, as well as for variable selection.  We prove that our proposed posterior concentrates around the true/sparse coefficient vector at the optimal rate, provide conditions under which the posterior can achieve variable selection consistency, and prove a Bernstein--von Mises theorem that implies asymptotically valid uncertainty quantification.  Computation of the proposed empirical Bayes posterior is simple and efficient, and is shown to perform well in simulations compared to existing Bayesian and non-Bayesian methods in terms of estimation and variable selection. 
\end{abstract}



\end{frontmatter}

\section{Introduction}
\label{S:intro}

Generalized linear models, or GLMs, which include normal, logistic, and Poisson regression as important special cases, are essential tools for data analysis in all quantitative fields; see, e.g., \citet{McCullaghNelder:1989} for a thorough introduction.  In modern applications, it is common for the number of predictor variables, $p$, to greatly exceed the sample size, $n$; this is the so-called ``$p \gg n$'' problem.  For example, logistic regression for presence/absence of a trait, with gene expression levels as covariates is one such problem. By now there is an enormous body of literature on the $p \gg n$ problem in the case of normal linear regression.  Popular methods include the lasso and its variants \citep{hastie.tibshirani.friedman.2009}. Bayesian efforts in the normal linear regression problem can be split into two categories: those based on shrinkage priors such as the {\em horseshoe} \citep{carvalho.polson.scott.2010, bhadra2019lasso, van2017uncertainty, bhadra.hsplus.2017, bhadra2016} and those based on spike-and-slab mixture priors \citep{castillo.vaart.2012, castillo.schmidt.vaart.reg, belitser2020empirical, george.mccullogh.1993}.  On the non-Bayesian side, a number of these methods have been extended from the normal linear model to other GLMs, e.g., the R package {\tt glmnet} \citep{glmnet} offers a comprehensive lasso-based toolkit, but the Bayesian developments in this direction are still limited; the methods tend to be tailored to logistic regression \citep{cao2020variable, narisetty2018skinny} and the theory focuses mostly on variable selection; the one exception, \citet{jeong2021posterior} gave results on posterior concentration rates in high-dimensional GLMs but did not address model selection or implementations of the Bayesian solutions they studied.  The main goal of the present paper is to offer a Bayesian (or at least Bayesian-like) solution to the high-dimensional GLM problem, having both strong theoretical support and an efficient numerical implementation that is not tailored to any one specific GLM. 

A challenge for Bayesian inference in high-dimensional models is that the priors for which posterior computations are relatively simple generally do not produce good theoretical posterior concentration properties and, vice versa, the priors with theoretical justification make posterior computations difficult and expensive.  For example, in normal linear regression, a computationally simple prior is a mixture of conjugate mean-zero normal priors on the coefficients, each component corresponding to a subset of active variables, but it has been shown \citep{castillo.vaart.2012} that the thin tails of the normal can lead to sub-optimal theoretical properties.  A theoretically better choice of prior is one with heavier, Laplace-type tails, but this added complexity translates to higher computational cost.  To overcome this obstacle, \cite{martin.mess.walker.eb} proposed the idea of using the data to properly center the prior.  The motivation is that the tails of the prior should not  matter if the prior is strategically centered, so then the computationally simpler conjugate normal priors could still be used.  Centering the model-specific conjugate normal priors on the corresponding least-squares estimators makes the approach {\em empirical Bayes}, in a certain sense, and the previous authors show that the corresponding empirical Bayes posterior has optimal asymptotic concentration properties and has strong empirical performance compared to existing Bayesian and non-Bayesian methods.  In other words, the double-use of data---in the prior and in the likelihood---does not hurt the method's performance in any way; in fact, one could argue that the double-use of data actually helps.  Beyond the normal linear model \citep{martin.mess.walker.eb, martin.tang.jmlr}, there is strong general theory in \cite{martin2019data} and promising results in applications, including \cite{ebmono, liu2019empirical, ebpiece}. 

The goal here is to develop the aforementioned empirical Bayes strategy for the case of high-dimensional GLMs. In Section~\ref{S:background}, we introduce the set up of the GLM problem and review the empirical Bayes approach for linear regression. In Section~\ref{S:eb}, we present our empirical Bayes GLM, including the particular choice of data-driven prior, the corresponding empirical Bayes posterior, and our proposed computational strategy.  The key challenge in the present GLM case compared to previous efforts in the linear model setting is that the models are sufficiently complicated that there is no conjugacy and, therefore, no posterior computations can be done in closed form.  Here the ``informativeness'' of the data-driven prior allows for some simple and accurate approximations.  In Section~\ref{S:theory}, we offer theoretical support for our proposed solution.  In particular, we have two basic kinds of posterior concentration results: those for the GLM coefficients, which are relevant to estimation, and those for the so-called configuration, or active set, which are relevant to variable selection.  In the former case, we give sufficient conditions for the posterior to concentrate around the true (sparse) coefficient vector at rates equivalent to those established in, e.g., \cite{jeong2021posterior}, which agree with the minimiax optimal rates in the linear model setting.  In the latter case, we give sufficient conditions (e.g., on the size of the smallest non-zero coefficient), comparable to those in \cite{narisetty2018skinny}, that ensure our marginal posterior for the configuration will concentrate on a set that contains exactly the true set of active variables.  While the results we obtain are similar to those found elsewhere in the literature, it is important to emphasize that, to our knowledge, there is no single Bayesian method for which both of these kinds of posterior concentration properties have been established.  And, again, the lack of conjugacy and closed-form expressions for (parts of) the posterior distribution creates some challenges compared to the linear model case, so some relatively new proof techniques---in-probability versus in-expectation bounds---are needed compared to \citet{martin.mess.walker.eb} and \citet{martin2019data}.  Beyond posterior concentration, we offer a Bernstein--von Mises theorem, comparable to that for linear models in \citet{castillo.schmidt.vaart.reg}, which establishes a large-sample Gaussian approximation of the posterior and, among other things, implies that the marginal posterior uncertainty quantification (e.g., credible intervals) are asymptotically valid. Section~\ref{S:examples} investigates the numerical performance of the proposed method in terms of estimation and variable selection in logistic and Poisson regression compared to existing methods. Finally, some concluding remarks are given in Section~\ref{S:discuss}; proofs are presented in the Appendix.

\section{Setup and background}
\label{S:background}

\subsection{Problem setup}

Suppose $y_1,\ldots,y_n$ are independent, where $y_i$ has density/mass function 
\[ f_{\eta_i}(y_i) \propto \exp\{y_i \eta_i - b(\eta_i)\}, \quad i=1,\ldots,n, \]
indexed by the real-valued natural parameters $\eta_1,\ldots,\eta_n$, where $b$ is a known, strictly convex function, i.e., $\ddot b(\eta) > 0$ for all $\eta$.  The interpretation of $b$ is that the expected value and variance of $y_i$ are $\dot b(\eta_i)$ and $\ddot b(\eta_i)$, respectively.  If the response $y_i$ has an associated vector of predictor variables $x_i \in \RR^p$, for $i=1,\ldots,n$, then introduce a coefficient vector $\theta \in \RR^p$ into the model via the relationship 
\[ (h \circ \dot b)(\eta_i) = x_i^\top \theta, \quad i=1,\ldots,n, \]
where $h$ is a given bijection called the {\em link function}.  The ``canonical'' link is $h = \dot b^{-1}$ and, in this case, the above simplifies to $\eta_i = x_i^\top \theta$.  But there are cases when non-canonical link functions $h$ are used and our theory here allows for this.  For example, in the Bernoulli data case, our results apply to both logistic regression (canonical link) and probit regression (non-canonical link).  

Write $f_\theta(y_i \mid x_i)$ for the density/mass function determined by the parameter vector $\theta$.  Then the likelihood and log-likelihood functions are, respectively, 
\[ L_n(\theta) = \prod_{i=1}^n f_\theta(y_i \mid x_i) \quad \text{and} \quad \ell_n(\theta) = \log L_n(\theta). \]
To ease the notation, set $\xi(\eta) = (h \circ \dot b)(\eta)$; if $h$ is the canonical link, then $\xi$ is the identity mapping.  Then the maximum likelihood estimator (MLE) $\hat{\theta}$ is the solution to the equation 
\begin{equation}
\label{eq:mle}
\dot\ell_n(\hat\theta) = 0 \iff \{y - \dot b(X\hat\theta)\}^\top \text{diag}\{ \dot\xi(X\theta) \} X = 0, 
\end{equation}
where $\text{diag}(\cdot)$ denotes the diagonal matrix determined by its vector argument, and $\dot\xi(X\theta)$ is the application of $\dot\xi$ to each entry in the vector $X\theta$; if $h$ is the canonical link, then $\dot\xi \equiv 1$.  The negative second derivative of the log-likelihood function is a matrix 
\[ J_n(\theta) = -\ddot \ell_n(\theta) = X^\top W(\theta) X, \]
where $W(\theta)$ is a diagonal matrix with entries 
\[ W_{ii}(\theta) = w(x_i^\top \theta) = \dot u(x_i^\top \theta) \dot\xi(x_i^\top \theta), \quad i=1,\ldots,n, \]
where $u = h^{-1}$ is the inverse link function.  When $h$ is the canonical link, we get $W_{ii}(\theta) = \ddot b(x_i^\top \theta)$. The observed Fisher information matrix is $J_n(\hat\theta) = X^\top W(\hat\theta) X$, the negative Hessian of the log-likelihood function evaluated at $\hat\theta$, which is positive definite.  For example, in binary regression with the canonical logit link, the $W$ matrix has diagonal entries 
\[ W_{ii}(\theta) = \frac{\exp(x_i^\top \theta)}{\{1+\exp(x_i^\top \theta)\}^2}, \quad i=1,\ldots,n, \]
which is bounded as a function of $\theta$. Similarly, in Poisson regression with the canonical log link, the $W$ matrix has diagonal entries
\[ W_{ii}(\theta) = \exp(x_i^\top \theta), \quad i=1,\ldots,n. \]

Our interest is in high-dimensional cases where the number of predictor variables $p$ exceeds the sample size $n$.  In such cases, the direct model fitting as described above cannot be done; intuitively, the data is not informative enough to reliably learn the very high-dimensional parameter $\theta$.  To side-step this obstacle, we shall assume, as is common, that the GLM is {\em sparse} in the sense that most of the $\theta$ coefficients are 0 (or at least negligible).  Then there is an ``active set'' of the predictor variables corresponding to the non-zero coefficient values, but this is unknown because $\theta$ itself is unknown.  To avoid overusing the term ``model'', here we will call the unknown active set of variables a {\em configuration} and denote it generically by $S$.  Since the configuration is unknown, it makes sense to decompose the unknown $\theta$ as $(S,\theta_S)$, where $S$ is a set of indices that corresponds to the ``active" coefficients in $\theta$ and $\theta_S$ is the vector of coefficients that correspond to a configuration $S$. Then, the above notation can be adjusted in a natural way.  That is, the likelihood and log-likelihood functions can be written as $L_n(S,\theta_S)$ and $\ell_n(S,\theta_S)$, respectively, the configuration-specific MLE is $\hat\theta_S$, the observed Fisher information is $J_n(S,\hat\theta_S)$, etc.  Throughout we will write $|S|$ for the cardinality of a configuration $S$, $\theta^\star$ for the true coefficients, and $S^\star$ for the true configuration; in some cases, we will write $s=|S|$ for the configuration size and, naturally, $s^\star = |S^\star|$ for the true sparsity level. Also, with a slight abuse of this notation, we will occasionally write $S(\theta) = \{j: \theta_j \neq 0\}$ to denote the configuration corresponding to a given coefficient vector. 

All of the results in \cite{narisetty2018skinny} and in \cite{cao2020variable} are established by focusing on configurations $S$ that are supersets of the true $S^\star$.  Since our asymptotic analysis goes beyond those in the aforementioned references, we will also need to consider configurations that are not supersets of $S^\star$, so some generalizations are in order.  The key observation is that, if $S \not\supset S^\star$, then $\hat\theta_S$ is {\em not} estimating $\theta_S^\star$.  Instead, $\hat\theta_S$ is estimating the minimizer of the Kullback--Leibler divergence which, in our case, is a solution to the equation 
\[ \{\dot b(X\theta^\star) - \dot b(X_S \theta_S)\}^\top \text{diag}\{\dot\xi(X_S\theta_S)\} X_S = 0. \]
Let $\theta_S^\dagger$ denote this solution.  Note that, first, this notation is slightly misleading, since there is no ``full vector'' $\theta^\dagger$ of which $\theta_S^\dagger$ is the $S$-specific sub-vector; instead, there is a different $|S|$-vector $\theta_S^\dagger$ for each $S$.  Second, if $S$ is a superset of $S^\star$, then $\theta_S^\dagger = \theta_S^\star$; this explains why \cite{narisetty2018skinny} do not need $\theta_S^\dagger$ in their superset-only analysis.  Lemmas~\ref{lem:Jn}--\ref{lem:mle} in Appendix~\ref{AA:likelihood} below generalize Lemmas~A1 and A3 in \cite{narisetty2018skinny} to cover the case where $\hat\theta_S$ is compared to $\theta_S^\dagger$ rather than $\theta_S^\star$.

\subsection{Empirical priors}
\label{SS:prior.background}

A very general construction of empirical or data-driven priors and the corresponding posterior concentration rate theory is given in \cite{martin2019data}.  There are elements of their general formulation used here in this application, but it is not necessary to review these for our purposes; see Appendix~\ref{AA:empirical.priors} for some of the relevant technical details.  For this reason, we focus our review here on the high-dimensional linear regression case investigated in \cite{martin.mess.walker.eb} and in \cite{martin.tang.jmlr}.

The normal linear model assumes that $y_i \sim \nm(x_i^\top \theta, \sigma^2)$, independent, for $i=1,\ldots,n$; for the discussion here, we take $\sigma$ to be known, but the case where $\sigma$ is unknown and assigned a prior has been considered in \cite{martin.tang.jmlr} and \cite{fang2023high}. As above, the idea is to reinterpret the high-dimensional coefficient vector $\theta$ as the pair $(S,\theta_S)$, the configuration and configuration-specific parameters.  Then the natural Bayesian approach to this would be to specify the prior hierarchically: first introduce a marginal prior for $S$, then a conditional prior for $\theta_S$, given $S$. 
\begin{enumerate}
\item For the marginal prior for $S$, the previous authors suggest the sparsity-encouraging prior mass function $\pi(S) \propto {p \choose |S|}^{-1} f_n(|S|)$, i.e., a rapidly decaying marginal prior mass function $f_n$ for the configuration size $|S|$ times a conditional uniform prior for $S$ of a given size $|S|$.  Here we consider the {\em complexity prior} that takes 
 \[ f_n(s) \propto p^{-\beta s}, \quad s=1,2,\ldots,s_{\text{max}}, \]
 where $s_\text{max}$ is a specified maximum complexity, which could be just the trivial choice $p$ or something smaller, such as $s_\text{max} = \text{rank}(X)$.  The hyperparameter $\beta > 0$ plays a crucial role and is discussed further below.  
 \item The conditional prior for $\theta_S$, given $S$, is where the data enters into the prior.  To avoid the sub-optimal behavior of the thin-tailed Gaussian prior while simultaneously retaining the computational convenience of the conjugate Gaussian form, \cite{martin.mess.walker.eb} suggested 
 \[ (\theta_S \mid S) \sim \nm_{|S|}\bigl( \hat\theta_S, \gamma (X_S^\top X_S)^{-1} \bigr), \quad S \subset \{1,2,\ldots,p\}, \]
 where $\hat\theta_S$ is the $S$-specific least squares estimator---which makes the prior data-dependent---and $\gamma > 0$ is a tuning parameter to be specified.  
 \end{enumerate} 

 This empirical prior is combined with the data-driven likelihood basically according to Bayes's theorem, resulting in a (data-dependent) probability distribution $\Pi^n$ that can be used for making inference on $\theta$.  In particular, the corresponding marginal posterior for the configuration $S=S(\theta)$ can be used for model selection purposes.  Computation is simple and fast, via an efficient Metropolis--Hastings-style Markov chain Monte Carlo procedure, thanks to the conjugacy of the data-centered empirical prior.  An associated R package \texttt{ebreg} \citep{tang2019package} is also available.
 
 The data-driven prior distribution, among other things, implies that this is not a genuinely ``Bayesian'' solution, so we cannot expect that it automatically inherits the good properties that Bayesian solutions typically enjoy.  But \cite{martin.mess.walker.eb} demonstrate theoretically that the posterior $\Pi^n$ achieves the adaptive, minimax optimal asymptotic concentration rate at the true/sparse $\theta^\star$.  In other words, there is no other posterior distribution---genuinely Bayesian or otherwise---that can concentrate around the true $\theta^\star$ any faster.  They also established that, under suitable conditions, the aforementioned marginal posterior distribution for the configuration concentrates its mass asymptotically on the true configuration, thus providing consistent model selection.  We establish similar properties here in this paper, but for the case of high-dimensional GLMs, so more details about the construction and the results will be given below.

\section{Empirical Bayes for high-dimensional GLMs}
\label{S:eb}

\subsection{Prior distribution}

As before, write the structured, high-dimensional coefficient vector $\theta$ as $(S,\theta_S)$, where the configuration $S$ is a subset of $\{1,2,\ldots,p\}$ and $\theta_S$ is a configuration-specific parameter that respects the structure determined by $S$.  Based on this decomposition, we follow \cite{martin.mess.walker.eb} and proceed with specification of the (empirical) prior $\Pi_n$ for $\theta$ hierarchically.  First, the marginal prior for $S$ has mass function 
\begin{equation}
\label{eq:complexity.prior}
\pi_n(S) \propto \textstyle{p \choose |S|}^{-1} p^{-\beta |S|}, \quad \text{$S \subset \{1,2,\ldots,p\}$ such that $|S| \leq s_n$}, 
\end{equation}
where $\beta > 0$ is a constant to be specified and $s_n$ is a deterministic, diverging sequence.  This is the same as the marginal prior for $S$ in Section~\ref{SS:prior.background}.  Second, the data-driven conditional prior for $\theta_S$, given $S$, is 
\begin{equation}
\label{eq:conditional.prior}
(\theta_S \mid S) \sim \Pi_{n,S} := \nm_{|S|}\bigl( \hat\theta_S, \gamma J_n(S,\hat\theta_S)^{-1} \bigr), 
\end{equation}
where $\gamma > 0$ is a constant to be specified. Above, the prior center is $\hat\theta_S$, the data-driven maximum likelihood estimator for the given $S$ and, similarly, $J_n(S,\hat\theta_S)$ is the corresponding data-driven observed Fisher information matrix.   So this empirical conditional prior has a similar form as that in Section~\ref{SS:prior.background} for the linear model case, but the specific pieces that go into it are just different here in the GLM setting.  

As indicated above, the conditional prior for $\theta_S$, given $S$, is data-driven in the sense that both the prior mean vector and covariance matrix depend on data $y$ through $\hat\theta_S$.  Also, recall that the diagonal entries of the information matrix $J_n(S,\hat\theta_S)$ are growing like $O(n)$, so the prior covariance matrix is rather small.  This is counter-intuitive when the prior center is a fixed constant, but makes perfect sense when the prior mean is data-driven.  That is, we ``believe in'' the data-driven prior center so the prior ought to be relatively tightly concentrated there; plus, if the prior covariance matrix were large, then there would be no point in/benefit to the data-driven prior centering. 

To summarize, the sparsity-encouraging empirical prior $\theta \sim \Pi_n$ is given by 
\begin{equation}
\label{eq:prior.big}
\Pi_n(d\theta) = \sum_S \pi_n(S) \, \nm_{|S|}(d\theta_S \mid \hat\theta_S, \gamma J_n(S,\hat\theta_S)^{-1}) \otimes  \delta_{0_{S^c}}(d\theta_{S^c}),
\end{equation}
where the sum is over all configurations $S$ supported by the prior mass function $\pi_n$ and $\delta_{0_{S^c}}(d\theta_{S^c})$ denotes a Dirac point mass differential term at the origin for the component $\theta_{S^c}$.  The empirical prior depends on the hyperparameters $(\beta,\gamma,s_n)$ which will be discussed in more detail below.  As is common, the theory offers some guidance on how to choose these hyperparameters, but not enough to fully determine their values.

\subsection{Posterior distribution}

If $L_n$ denotes the GLM's likelihood, and $\Pi_n$ the empirical prior in \eqref{eq:prior.big} with the mixture form, then the proposed posterior distribution for $\theta$ is defined as 
\begin{equation}
\label{eq:post}
\Pi^n(d\theta) = \frac{L_n(\theta)^\alpha \, \Pi_n(d\theta)}{\int_{\RR^p} L_n(\vartheta)^\alpha \, \Pi_n(d\vartheta)}, \quad \theta \in \RR^p, 
\end{equation}
where $\alpha \in (0,1)$ is a fixed constant that can be chosen arbitrarily close to 1.  

Our proposed use of a power-likelihood in the Bayesian updating might make some readers uncomfortable, so some remarks on this are in order.  First, with our use of a data-driven prior, the lines between the ``likelihood part'' and ``prior part'' in Bayes's formula have been blurred.  So, one can easily adjust the above formula so that it is the ordinary likelihood $L_n$ combined with a slightly modified empirical prior $\widetilde\Pi_n(d\theta) \propto L_n(\theta)^{-(1-\alpha)} \, \Pi_n(d\theta)$, and get exactly the same posterior.  In our opinion, the version in \eqref{eq:post} is preferred because it is more transparent.  Second, the original motivation for choosing $\alpha < 1$ \citep[e.g.,][]{martin.walker.eb} was to prevent the posterior from tracking the data too closely as a result of its double-use of the data; this is discussed extensively in, e.g., \cite{walker.hjort.2001} and \cite{walker.lijoi.prunster.2005a}.  Relatively recent evidence suggests that a choice of $\alpha < 1$ is not necessary for good posterior concentration properties \citep[e.g.,][]{belitser.nurushev.uq, belitser2020empirical}.  But what is needed to accommodate $\alpha=1$ adds significant technical complications without any benefits in terms of faster rates, etc.  Indeed, in the case of GLMs, \cite{jeong2021posterior} pointed out that there are non-trivial differences in the strength of their theoretical results for $\alpha=1$ versus $\alpha < 1$; in particular, much stronger conditions are required in some cases to get the same concentration rates using $\alpha=1$ compared to $\alpha < 1$.  Finally, there may even be some practical benefits to the use of power-likelihoods, e.g., in terms of robustness \citep[e.g.,][]{miller.dunson.power, syring.martin.scaling, syring.martin.subexp}, ``safety'' \citep[e.g.,][]{grunwald.ommen.scaling, grunwald.mehta.rates}, and/or uncertainty quantification \citep[e.g.,][]{martin.tang.jmlr, martin2020empirical}.

Back to the task at hand, thanks to the parametrization $\theta = (S,\theta_S)$ and the hierarchical prior, a marginal posterior distribution for $S$ is available, i.e., 
\begin{align*}
\pi^n(S) & = \frac{\pi_n(S) \int_{\RR^{|S|}} L_n(S,\theta_S)^\alpha \, \nm_{|S|}(\theta_S \mid \hat\theta_S, \gamma J_n(S,\hat\theta_S)^{-1}) \, d\theta_S}{\sum_R \pi_n(R) \int_{\RR^{|R|}} L_n(R,\theta_R)^\alpha \, \nm_{|R|}(\theta_R \mid \hat\theta_R, \gamma J_n(R,\hat\theta_R)^{-1}) \, d\theta_R} \\
& \propto \pi_n(S) \int_{\RR^{|S|}} L_n(S,\theta_S)^\alpha \, \nm_{|S|}(\theta_S \mid \hat\theta_S, \gamma J_n(S,\hat\theta_S)^{-1}) \, d\theta_S,
\end{align*}
where $x \mapsto \nm(x \mid \mu, \sigma^2)$ denotes a $\nm(\mu, \sigma^2)$ density.  Although there is generally no closed-form expression for the last integral above, a formal Laplace approximation gives a nice, simple expression:
\begin{equation}
\label{eq:S.post}
\pi^n(S) \propto \pi(S) \, (1 + \alpha \gamma)^{-|S|/2} \, L_n(S,\hat\theta_S)^\alpha, \quad \text{all large $n$}. 
\end{equation}
\cite{cao2020variable} also employ a Laplace approximation, though their expression is different because their prior is different.  Of course, it is too much to expect that this approximation be accurate simultaneously across {\em all} configurations, but we do not need such a strong result.  It is enough that this approximation be accurate over a class of relatively small configurations \citep[e.g.,][]{shun1995laplace, barber2016laplace}.  This class is described in Section~\ref{S:theory} below, and a precise result on the Laplace approximation's accuracy is given in Lemma~\ref{lem:laplace} in Appendix~\ref{AA:marginal}.  Details on posterior computation are given next.

\subsection{Computation}
\label{SS:computation}

If variable selection is the goal, then focus is on the marginal posterior for $S$.  We propose to use the Laplace approximation of $\pi^n(S)$ in \eqref{eq:S.post} in a simple Metropolis--Hastings Markov chain Monte Carlo scheme.  Note that this does not require that we can evaluate the normalizing constant implicit in \eqref{eq:S.post}.  

Given a proposal function $q(S' \mid S)$, one iteration of the Metropolis--Hastings algorithm is as follows:
 \begin{enumerate}
\item Given a current state $S$, sample $S'\sim q(\cdot \mid S)$.
\item Go to the new state $S'$ with probability 
\[\min\Bigl\{1, \frac{\pi^n(S)}{\pi^n(S')} \frac{q(S' \mid S)}{q(S \mid S')}\Bigr\},\]
where $\pi^n(S)$ is defined in \eqref{eq:S.post}. Otherwise, stay in the current state $S$.
\end{enumerate}
We use a proposal distribution that is symmetric, one that samples $S'$ uniformly from those that differ from $S$ in exactly one position. This simplifies our computations as the $q$-ratio above is simply 1.  This process is repeated $M$ times, which yields a sample of configurations $S^{(1)}, \ldots, S^{(M)}$ from our posterior distribution $\pi^n(S)$; this is after a burn-in period that excludes the first 20\% of the samples generated.  This is relatively efficient, since the likelihood-based ingredients---the MLE $\hat\theta_S$, the Fisher information $J_n(S,\hat\theta_S)$, etc.---only need to be evaluated for the configurations $S$ that are selected in the MCMC.  

For variable selection, a relevant quantity is the {\em inclusion probability} associated with each candidate variable $j=1,\ldots,p$ that could be included.  The simplest way to express this is as the posterior probability that coefficient $\theta_j$ attached to variable $X_j$ is non-zero.  With a slight abuse of notation, we will refer to the inclusion probability for variable $j$ as $\pi^n(j)$, which is 
\begin{equation}
\label{eq:ip}
\pi^n(j) := \pi^n(\{S: S \ni j\}) \approx \frac{1}{M} \sum_{m=1}^M 1\{S^{(m)} \ni j\}, 
\end{equation}
where the right-hand side is the Monte Carlo approximation, the proportion of those configurations $S^{(m)}$'s drawn that include variable $j$.  From here, a natural variable selection procedure would include all those variables for which the inclusion probability exceeds a specified threshold; see Section~\ref{S:examples} below.  

If there is also interest in estimation of/inference on the coefficients $\theta$, then one can easily augment the above Monte Carlo scheme by inserting a rejection sampling step where, given $S$, the corresponding coefficient $\theta_S$ is drawn from the corresponding conditional posterior.  Thanks to the empirical prior structure, a simpler alternative would be to average the $S$-specific MLEs with respect to the the posterior distribution $\pi^n$ of $S$, which is akin to the exponential aggregation in, e.g., \citet{rigollet.tsybakov.2012}. Similarly, if the goal is prediction of a new response $\tilde y$ associated with a new set of covariate values $\tilde X$, then a third step is added where in a draw is made from the posited model given $(S,\theta_S, \tilde X)$. 

When $S$ is the sole focus, alternatively, a shotgun stochastic search (SSS) algorithm can be employed, as in Algorithm 1 of \cite{cao2020variable}. 
In SSS, models that are neighbors of a selected model are evaluated, and then the next chosen model is sampled from these neighbors proportional to their posterior probabilities, repeating the process. This is more efficient than the aforementioned Monte Carlo strategy in effectively exploring the configuration space. In practice, however, we found that our method with $M=10^4$ posterior samples (with a 20\% burn-in), computes significantly faster; this is likely due to SSS having to compute posterior probabilities for all the neighboring models.

\section{Asymptotic properties}
\label{S:theory}

\subsection{Setup and conditions}

\cite{narisetty2018skinny} and \cite{cao2020variable} investigate certain asymptotic properties of their proposed posterior for $\theta$, but they focus exclusively on (a)~logistic regression and (b)~results concerning the marginal posterior $\pi^n$ for the configuration $S$.  Here we extend the analysis beyond the logistic regression case to arbitrary GLMs as described above, with arbitrary link functions, and establish conditions under which our proposed posterior distribution $\Pi^n$ for $\theta$ concentrates around the true $\theta^\star$ at (nearly) the optimal rate \citep[e.g.,][]{rigollet2012}, adaptive to the unknown sparsity level $|S(\theta^\star)|$.  

Below are two conditions crucial to the developments here and in \cite{narisetty2018skinny} and \cite{cao2020variable}.  These can be roughly classified as conditions on the dimension of the problem and on the design matrix $X$.  The third condition concerns the hyperparameters in our empirical prior.  

\begin{cond}
\label{cond:dim}
$p=p_n \to \infty$ and $\log p = o(n)$ as $n \to \infty$.
\end{cond}

\begin{cond}
\label{cond:design}
There exists $K > 0$, $\lambda > 0$, $\tau \in [0,1]$, and $\tau' \in (\tau,1]$ such that 
\begin{itemize}
\item[(a)] the entries in $X$ are bounded in absolute value by $K$
\item[(b)] if $\lambda_\text{min}$ and $\lambda_\text{max}$ are operators that return the smallest and largest eigenvalues of their arguments, respectively, then
\begin{equation}
\label{eq:eigenvalue}
\lambda \leq \min_{S: |S| \leq |S^\star| + s_n} \lambda_\text{min}\{ n^{-1} J_n(S,\theta_S^\dagger) \} \leq \Lambda_{|S^\star|+s_n} \leq K^2 (n / \log p)^\tau, 
\end{equation}
where 
\[ \Lambda_k = \max_{S: |S| \leq k} \lambda_\text{max}\{n^{-1} J_n(S,\theta_S^\dagger)\}, \quad k=1,2,\ldots, \]
and 
\begin{equation}
\label{eq:s.n}
s_n = O\bigl( (n / \log p)^{(1-\tau')/2} \bigr), \quad n \to \infty. 
\end{equation}
\end{itemize} 
\end{cond}

\begin{cond}
\label{cond:prior}
The power $\alpha > 0$ in \eqref{eq:post} is strictly less than 1. Also:
\begin{itemize}
\item[(a)] the prior hyperparameter $\gamma > 0$ in \eqref{eq:conditional.prior} satisfies 
\begin{equation}
\label{eq:gamma}
\gamma = O(\Lambda_{2 s_n}^2), \quad n \to \infty, 
\end{equation}
where $\Lambda_s$ and $s_n$ are as defined above, and 
\item[(b)] the prior hyperparameter $\beta > 0$ in \eqref{eq:complexity.prior} is such that $p^{\beta - \kappa} > s_n$, where $\kappa > 1$ is the constant specified in Lemma~\ref{lem:loglik.diff} and $s_n$ is as in \eqref{eq:s.n}.  
\end{itemize}
\end{cond}

Conditions~\ref{cond:dim} and \ref{cond:design}(a) are standard in the high-dimensional inference literature.  The lower bound in Condition~\ref{cond:design}(b) is a type of restricted eigenvalue condition, similar to those assumed in \cite{narisetty2018skinny} and \cite{cao2020variable}.  The upper bound weakens and generalizes the ``bounded eigenvalue condition'' in, e.g., \cite{bondell.reich.2012}.  Since $J_n(S,\theta_S^\dagger) = X_S^\top W(S,\theta_S^\dagger) X_S$, it is clear that if $n^{-1} X_S^\top X_S$ has bounded eigenvalues, and if the entries of $W$ are uniformly bounded, as would be the case in Examples~5--8 in \cite{jeong2021posterior}, including logistic regression, then the upper bound in \eqref{eq:eigenvalue} holds with $\tau=0$.  Such cases correspond to the weakest constraint \eqref{eq:s.n} on the support for $S$, since we can take $\tau'=0$ too.  More generally, if $n^{-1} X_S^\top X_S$ has bounded eigenvalues, then $\Lambda_{|S|}$ is bounded by the maximum entry on the diagonal of $W(S,\theta_S^\dagger)$.  Since the diagonal entries are typically increasing functions of their arguments, the bound is $w(\|X_S \theta_S^\dagger\|_\infty)$.  Since the focus is on large configurations, and since $\|X_S \theta_S^\dagger\|_\infty = \|X_S \theta_S^\star\|_\infty = \|X\theta^\star\|_\infty$ for $S \supset S^\star$, the upper bound in \eqref{eq:eigenvalue} is only slightly stronger than assuming ``$w(\|X\theta^\star\|_\infty) \lesssim (n/\log p)^\tau$''.  The Poisson log-linear model is one of the most challenging examples, where $h = \dot b^{-1}$ and $\xi$ is the identity, so $w(\eta) = e^\eta$.  For our Condition~\ref{cond:design} to be met in the Poisson case, we would roughly need $\theta^\star$ to satisfy 
\[ \|X\theta^\star\|_\infty \lesssim \log\{(n/\log p)^\tau\} = O(\log n). \]
When $p$ is polynomial in $n$, this restriction is equivalent to that in Remark~2 of \cite{jeong2021posterior}; when $\log p$ is a small power of $n$, the above restriction is stronger than theirs.  But our sometimes-stronger condition here allows for a more extensive asymptotic analysis, i.e., results on the marginal posterior for $S$. 

For Condition~\ref{cond:prior}(a), the constant $\tau$ in \eqref{eq:eigenvalue} is determined by $X$, so, theoretically, it is not impossible to determine $\tau$ and to set $\gamma$ in \eqref{eq:gamma} accordingly.  For example, if $X$ is such that $s \mapsto \Lambda_s$ is uniformly bounded, then \eqref{eq:eigenvalue} holds with $\tau=0$ and then $\gamma$ can be taken as a constant too.  When $\tau > 0$, the corresponding $\gamma$ is a diverging sequence in $n$.  Recall that the primary part of the $(\theta_S \mid S)$ covariance matrix is $J_n(S,\hat\theta_S)^{-1}$, which itself is $O(n^{-1})$.  So, multiplying this by  $\gamma \to \infty$ still allows the prior mass to be concentrating around the data-driven center $\hat\theta_S$, just slower.  Finally, as argued in \cite{narisetty2018skinny}, the constant $\kappa$ can be chosen arbitrarily close to 1, so ``$\beta > \kappa$,'' which is implied by Condition~\ref{cond:prior}(b), is just a little stronger than ``$\beta > 1$.'' Indeed, since $s_n$ is growing strictly slower than $n^{1/2}$, it suffices to take $\beta - \kappa$ greater than $\frac12 (\log n)/(\log p) \leq \frac12$. 


Below we present two different types of results.  The first type concerns generally how the posterior distribution $\Pi^n$ for the coefficient vector $\theta$ concentrates its mass around the sparse true value $\theta^\star$.  The second type concerns how the marginal posterior mass function $\pi^n$ for $S$ concentrates around its true value $S^\star$, where $S^\star = S(\theta^\star)$ is the configuration corresponding to the true $\theta^\star$.  The proofs, presented in Appendix~\ref{S:proofs}, follow those in \citet{martin.mess.walker.eb} and \citet{martin.walker.eb} relatively closely.  The key difference is that, here, we cannot get closed-form expressions for the posterior, so suitable in-expectation bounds as in the above references are out of reach.  The proofs here are based on in-probability bounds, which are more flexible, so our general strategies here might be applicable in other situations too.

Although the Gaussian linear model is a GLM, it is possible to derive equivalent results for this model directly and under weaker conditions \citep{martin.mess.walker.eb}.  So the machinery presented below is intended only for the genuinely more complex GLM setting.

\subsection{Posterior concentration results}

As a first result, we consider concentration of the full posterior for $\theta$ around the true, sparse coefficient vector $\theta^\star$ under a statistically universal metric, namely, the Hellinger distance between joint distributions of $y$ determined by the true $\theta^\star$ and by a generic $\theta$.  More specifically, if $p_\theta(y \mid x)$ denotes the distribution of (scalar) $y$, given covariate vector $x$ and coefficient vector $\theta$, then define the (expected) Hellinger distance $H_n$ as 
\begin{equation}
\label{eq:hellinger}
H_n(\theta^\star, \theta) = \Bigl[ \frac1n \sum_{i=1}^n \int \{ p_\theta(y_i \mid x_i)^{1/2} - p_{\theta^\star}(y_i \mid x_i)^{1/2} \}^2 \, dy_i \Bigr]^{1/2}. 
\end{equation}
This is just the expected squared Hellinger distance between marginals, where expectation is with respect to the empirical distribution of the rows $x_i$ in the matrix $X$.  Note that $H_n$ depends on $X$. 

\begin{thm}
\label{thm:h.rate}
Under Conditions~\ref{cond:dim}--\ref{cond:prior}, the posterior $\Pi^n$ defined above satisfies 
\begin{equation}
\label{eq:h.rate}
\sup_{\theta^\star: |S(\theta^\star)| \leq s_n} \E_{\theta^\star} \Pi^n(\{\theta: H_n(\theta^\star, \theta) > M \eps_n(\theta^\star) \}) \to 0, \quad n \to \infty, 
\end{equation}
where $\eps_n^2(\theta^\star) = n^{-1} s^\star \log p$, with $s^\star=|S(\theta^\star)|$ and $M > 0$ a  large constant. 
\end{thm}

This is the same rate established in Theorem~2 of \cite{jeong2021posterior} for a class of Bayesian posterior distributions based on priors that do not depend on data.  This is also effectively the minimax optimal rate, i.e., $\{n^{-1} s^\star \log(p/s^\star)\}^{1/2}$, in sparse, high-dimensional linear regression that is attained by \citet{abramovich.grinshtein.2010}, \citet{abramovich2016model}, \citet{ariascastro.lounici.2014}, \citet{martin.mess.walker.eb}, and \citet{belitser2020empirical}.  The difference between the rate in Theorem~\ref{thm:h.rate} and the minimax optimal rate is negligible since $p \gg s^\star$ and, consequently, $\log(p/s^\star) \sim \log p$. It is also important to recognize that the rate achieved is optimal corresponding to the unknown, true sparsity level $|S(\theta^\star)|$.  Since the method itself has no knowledge of this sparsity level, we say that the minimax optimal rate is achieved {\em adaptively}.  

Admittedly, the Hellinger rate is not so easily interpretable, but rates under different metrics are possible.  For a more interpretable result (modulo more complicated conditions), we can appeal to the arguments given in the proof of Theorem~3 in \citet{jeong2021posterior}.  Their proof shows that a Hellinger rate like in Theorem~\ref{thm:h.rate} above and an effective-dimension bound like in Theorem~\ref{thm:dim} below together imply a rate in terms of other distances, including the $\ell_2$-distance on $\theta$.  Their argument is not for a specific kind of posterior distribution, so what works for them in their case works equally well for us here.  The following corollary makes our claims precise.

\begin{cor}
\label{cor:l2.rate}
Under Conditions~\ref{cond:dim}--\ref{cond:prior}, the posterior $\Pi^n$ defined above satisfies 
\begin{equation}
\label{eq:theta.post.rate}
\sup_{\theta^\star: |S(\theta^\star)| \leq s_n} \E_{\theta^\star} \Pi^n\Bigl( \Bigl\{\theta: \|\theta-\theta^\star\|_2^2 > \frac{M\eps_n^2(\theta^\star) }{\phi^2(K|S(\theta^\star)|, W(\theta^\star))} \Bigr\} \Bigr) \to 0, \quad n \to \infty, 
\end{equation}
where $\eps_n^2(\theta^\star) = n^{-1} |S(\theta^\star)| \log p$, $M > 0$ and $K > 0$ are constants, and 
\[ \phi(s, W) = \inf_{\theta: 1 \leq |S(\theta)| \leq s} \frac{\|W^{1/2} X \theta\|_2}{n^{1/2} \|\theta\|_2}, \]
denotes the smallest $s$-sparse singular value of $X^\top W X$.
\end{cor} 

The appearance of an additional term---the sparse singular value---depending on $X$ is expected since the response $y$ depends directly on $X\theta$, not on $\theta$ itself.  This is easy to see in the linear model case where the Hellinger distance is proportional to the $\ell_2$-norm between fitted values.  So to strip the $X$ away and investigate the posterior concentration directly in terms of $\theta$ requires some conditions on $X$, which are baked into the effect the $\phi$ term has on the rate.  For example, if the $\phi$ term in \eqref{eq:theta.post.rate} is bounded away from 0, which amounts to a condition on $X$, then that term can be absorbed into the constant $M$ and the $\ell_2$-rate agrees with the Hellinger rate above.  In any case, the result here is the same as that proved in \cite{jeong2021posterior}, so the reader interested in details about $\phi$ can refer to their discussion. 

That the posterior for $\theta$ concentrates at the (near) optimal rate for sparse $\theta^\star$ true vector {\em suggests} that the posterior for $S$ is concentrating on the true $S^\star = S(\theta^\star)$, but this is not a consequence of Theorem~\ref{thm:h.rate}.  The asymptotic behavior of the $S$-posterior must be investigated directly.  The first such result concerns the ``effective dimension'' of the posterior distribution for $\theta$ is not much larger than that of $\theta^\star$.  In other words, $\pi^n$ concentrates on $S$ with $|S|$ that are smaller than a multiple of $|S^\star|$, where $S^\star = S(\theta^\star)$.  

\begin{thm}
\label{thm:dim}
Under Conditions~\ref{cond:dim}--\ref{cond:prior}, for any $C > (1 - \alpha \kappa / \beta)^{-1} > 1$, 
\[ \sum_{S: |S| > C|S^\star|} \pi^n(S) \to 0, \quad \text{in $\prob_{\theta^\star}$-probability}, \]
for all $\theta^\star$ such that $|S^\star| \leq s_n$. 
\end{thm}

That the constant $C$ above is greater than 1 follows from the fact that $\beta > \kappa$ and $\alpha < 1$.  Again, the take-away message here is that the posterior distribution for $\theta$ is concentrating on a space that is genuinely low-dimensional and, in particular, is of dimension not much greater than $|S^\star|$. Theorem~\ref{thm:dim} also {\em suggests} that $\pi^n$ is concentrating on $S^\star$, but this is not a direct consequence.  Towards this, we have one more result which states that $\pi^n$ tends to not over-fit, i.e., it tends to avoid supersets of $S^\star$. 

\begin{thm}
\label{thm:no.super.sets}
Under Conditions~\ref{cond:dim}--\ref{cond:prior},
\[ \sum_{S: S \supset S(\theta^\star)} \pi^n(S) \to 0 \quad \text{in $\prob_{\theta^\star}$-probability}, \]
for all $\theta^\star$ such that $|S(\theta^\star)| \leq s_n$. 
\end{thm}

Virtually the same argument used to prove Theorem~\ref{thm:no.super.sets} can be used to conclude that $\pi^n$ will not concentrate on any $S$ that contains an unimportant variable.  To ensure that $\pi^n$ does not concentrate on models that exclude at least one important variable, it is necessary to assume that the non-zero coefficients attached to the important variables are not too small.  The intuition is that, if an important variable is ``just barely'' important, the data may not be informative enough to detect it given the relatively strong penalty on model complexity.  The following result imposes a version of the familiar {\em beta-min condition} \citep{buhlmann.geer.book, buhlmann2011comment} to ensure that the signals are large enough to be detected; see \eqref{eq:beta.min}.   

\begin{thm}
\label{thm:consistent}
Under Conditions~\ref{cond:dim}--\ref{cond:prior}, $\pi^n\{S(\theta^\star)\} \to 1$ with $\prob_{\theta^\star}$-probability $\to 1$ for all $\theta^\star$ such that $c|S(\theta)^\star| \leq s_n$ and 
\begin{equation}
\label{eq:beta.min}
\min_{j \in S(\theta^\star)} \theta_j^{\star 2} \geq c n^{-1} |S(\theta^\star)| \Lambda_{c|S(\theta^\star)|} \log p, 
\end{equation}
for some constant $c > 1$. 
\end{thm}

The condition \eqref{eq:beta.min} is exactly the same as in \cite{narisetty2018skinny} and \cite{cao2020variable}. It is also equivalent to the beta-min condition for lasso variable selection consistency when $w=0$, and is slightly stronger when $w>0$.

\subsection{Posterior distribution approximation}

Aside from knowing that the posterior distribution for $\theta$ concentrates around the true $\theta^\star$ sufficiently fast, there is interest in knowing the approximate form or shape of that limiting posterior.  In particular, this helps to demonstrate that uncertainty quantification about $\theta$ (e.g., credible sets) derived from the posterior distribution is at least asymptotically valid in a frequentist sense.  Along these lines, under the same conditions as in the previous theorems, Theorem~\ref{thm:bvm} below establishes a {\em Bernstein--von Mises} result comparable to that in Corollary~2 of \cite{castillo.schmidt.vaart.reg} for a high-dimensional linear regression model.  To our knowledge, this is the first posterior normality result in the literature on sparse, high-dimensional GLMs. 

Towards this, define the Gaussian approximation 
\begin{equation}
\label{eq:gaussian.limit}
\Psi^n(d\theta) = \nm_{|S^\star|}\bigl( d\theta_{S^\star} \mid \hat\theta_{S^\star}, \varrho \, J_n(S^\star, \hat\theta_{S^\star})^{-1} \bigr) \otimes \delta_{0^{S^{\star c}}}(d\theta_{S^{\star c}}), 
\end{equation}
where the variance multiplier $\varrho$ depends only on the inputs $(\alpha,\gamma)$ as follows: 
\[ \varrho = \gamma (1 + \alpha\gamma)^{-1}. \]
This is effectively the Gaussian posterior approximation that the data analyst would use if he/she {\em knew} the configuration $S^\star$ and, in particular, the marginal posterior credible intervals for $\theta_j$, with $j \in S^\star$, would virtually agree with the classical likelihood-based confidence intervals returned from, say, R's {\tt glm} function for large $n$; more on these points below. Theorem~\ref{thm:bvm} below states that our proposed posterior distribution $\Pi^n$ asymptotically resembles the near-oracle Gaussian approximation \eqref{eq:gaussian.limit} in the total variation sense. 

\begin{thm}
\label{thm:bvm}
Under Conditions~\ref{cond:dim}--\ref{cond:prior}, for $\Psi^n$ as in \eqref{eq:gaussian.limit}, 
\[ \E_{\theta^\star} \, d_\text{\sc tv}\bigl( \Pi^n, \Psi^n \bigr) \to 0, \]
for any $\theta^\star$ such that \eqref{eq:beta.min} holds and 
\begin{equation}
\label{eq:small.model}
n^{-1} |S(\theta^\star)|^3 \Lambda_{|S(\theta^\star)|} \log p \to 0. 
\end{equation}
\end{thm}

Recall that Condition~\ref{cond:prior} requires that, in most cases, $\gamma=\gamma_n$ is a diverging sequence, i.e., $\gamma \to \infty$ as $n \to \infty$.  In that case, $\varrho=\varrho_n$ also depends on $n$ and satisfies $\varrho \to \alpha^{-1}$ as $n \to \infty$.  Even if $\gamma$ is allowed to be a constant, we can choose that constant to make $\varrho$ as close to $\alpha^{-1}$ as we like.  Since $\alpha \in (0,1)$ is a fixed constant that can be (and is) taken arbitrarily close to 1, the limiting multiplier $\alpha^{-1}$ need only be slightly greater than 1.  So, $\Psi^n$ in \eqref{eq:gaussian.limit} is effectively the {\em oracle} Gaussian approximation and, likewise, our posterior uncertainty quantification is asymptotically both valid and effectively efficient. 

Condition \eqref{eq:small.model} amounts to assuming that the true $\theta^\star$ is a bit lower complexity compared to what was needed for posterior concentration rates, etc.  This restriction is not too surprising, since the result is closely tied to the accuracy of Laplace's approximation, which is only expected in relatively low-complexity cases.  Indeed, \cite{shun1995laplace} show that Laplace approximations are expected to be accurate for integrals over spaces of dimension $o(n^{1/3})$, which is precisely the range of $|S(\theta^\star)|$ that would satisfy \eqref{eq:small.model}; see, also, \cite{barber2016laplace}.  The restriction \eqref{eq:small.model} can also be compared to the corresponding result in \cite{castillo.schmidt.vaart.reg}.  For one thing, in the case of a Gaussian likelihood, there is already a strong, built-in nudge towards a Gaussian posterior, so it makes sense that these authors could prove a result comparable to that in Theorem~\ref{thm:bvm} under weaker conditions.  They also adopt a ``small lambda regime''---referring to the rate parameter in their Laplace prior for non-zero coefficients---which offers additional flexibility to handle more complex configurations.  All this being said, we make no claims that \eqref{eq:small.model} is necessary for our proposed posterior to enjoy a Bernstein--von Mises theorem---improvements on some of the bounds used in our analysis may be available.

\section{Numerical results}
\label{S:examples}

\subsection{Methods and metrics}

Our focus here is on variable selection and estimation performance of our proposed empirical Bayes method compared to existing methods in two different GLM contexts.  There is a plethora of literature \citep[e.g.,][]{narisetty2018skinny, cao2020variable, bhadra2019lasso, wei2020contraction} that focuses on high-dimensional logistic regression, but a dearth of literature on Bayesian methods---with numerical implementations---that can handle GLMs besides logistic regression.  For this reason, we showcase our method's performance in both {\em logistic} regression and {\em Poisson} regression (with canonical log link). 

We start with some details about the competitor methods we consider, including lasso, adaptive lasso, SCAD, MCP, horseshoe, and skinnyGibbs. Most of these methods were implemented via R packages: lasso and adaptive lasso with \texttt{glmnet} and SCAD and MCP through \texttt{ncvreg}. No R package is available for the horseshoe in GLMs, so we relied instead on STAN. For logistic regression, the horseshoe method is coded using the \texttt{rstan} package, and hyperparameters were chosen based on recommendation of \cite{piironen2017sparsity} and \cite{bhadra2019lasso}. Specifically, we use the regularized horseshoe prior with $c^2 \sim {\sf InvGamma}(2, 8)$, and $\tau$ determined based on the number of effective nonzero coefficients. The MCMC for the horseshoe was run with two chains and 5,000 posterior draws for each chain. Posterior samples were obtained for the coefficient vector $\theta$, and since our interest is in the task of variable selection, we follow the default procedure that deems a variable as ``active'' if its 95\% posterior credible interval does not include zero. For Poisson regression, there was not as much guidance in the literature on implementation; we used the \texttt{rstanarm} package \citep{goodrich2018rstanarm}, but these results are not reported here as the horseshoe was computationally restrictive in terms of runtime when running simulations with a large number of replications. SkinnyGibbs was implemented with R package \texttt{skinnybasad}, directly obtained from the authors. Since this method was developed exclusively for variable selection in logistic regression, we do not present results for skinnyGibbs in the Poisson regression setting. 

We implement the proposed empirical Bayes solution using the Metropolis--Hastings strategy in Section~\ref{SS:computation}.  We also tried the shotgun stochastic search as discussed there, but we found that this produced similar results to Monte Carlo with no appreciable gain in computational efficiency; so we opted for the simpler of the two.  After a burn-in period in our sampling scheme, we return $M=10^4$ posterior samples of the configuration $S$, from which we can evaluate the inclusion probabilities as defined in \eqref{eq:ip}.  For a variable selection procedure based on our empirical Bayes solution, we propose 
\[ \hat S = \hat S(t) = \{j: \pi^n(j) > t\}, \]
where $\pi^n(j)$ is the (Monte Carlo approximation of) inclusion probability in \eqref{eq:ip} and $t \in (0,1)$ is a user-specified threshold.  We investigate the performance of our proposed method with two different choices of the threshold $t$, namely, $t=0.1$ and $t=0.5$; we refer to these below as EB1 and EB2, respectively.  The latter threshold 0.5 corresponds to selecting the so-called ``median probability model'' advocated for in, e.g., \cite{barbieri.berger.2004}.  The former threshold 0.1 is designed to select a larger subset of variables and is motivated by Bayesian decision-theoretic considerations where a Type~II error (missing an important variable) is taken to be 9 times as costly as a Type~I error.

For comparing the performance of the different variable selection methods, we report three different metrics: sensitivity (TPR), specificity (TNR), and Matthew's correlation coefficient (MCC). Sensitivity, or true positive rate, allows us to see how well the method does in identifying true signals or genuinely non-zero coefficients; specificity, or true negative rate, shows the method's ability to correctly identify the noise or genuinely zero coefficients; and MCC looks at all the four categories of the confusion matrix and combines them into one single metric. An MCC of 1 means the method perfectly distinguished signal from noise, while an MCC of 0 means the signal discovery was no better than random guessing. These metrics are defined as 
\begin{align*}
\text{TPR} & = \frac{\text{TP}}{\text{TP}+\text{FN}} \qquad \qquad \text{TNR} = \frac{\text{TN}}{\text{TN}+\text{FP}} \\
\text{MCC} & = \frac{\text{TP}\times \text{TN}- \text{FP}\times \text{FN}}{\{ (\text{TP}+\text{FP})(\text{TP}+\text{FN})(\text{TN}+\text{FP})(\text{TN}+\text{FN})\}^{1/2}},
\end{align*}
where, TP, TN, FP, and FN are the number of true positives, true negatives, false positives, and false negatives, respectively, and for MCC, we adopt the convention $0/0 \equiv 0$.  The relevance of these metrics depends on the data analyst's priorities.  For example, if Type~II errors are more costly than Type~I, then TPR would be the relevant metric.  On the other hand, if the two errors cost roughly the same, then MCC would be a more relevant metric.  MCC has been shown to be more reliable than other measures that are commonly used, including area under the receiver operating characteristic curve, balanced accuracy, bookmaker informedness, and markedness \citep{chicco2023matthews}.


For logistic regression, we focused on variable selection since skinnyGibbs, the most competitive method in the Bayesian landscape, is designed only for logistic regression variable selection. For Poisson regression, in addition to variable selection, we also compare estimation performance of our method to the others (lasso, adaptive lasso, SCAD, and MCP). Neither horseshoe nor skinnyGibbs are included in the Poisson regression comparisons---horseshoe is too restrictive in terms of runtime and no version of skinnyGibbs is currently available for GLMs besides logistic regression. For estimation, the metric we consider is mean squared error (MSE).

\subsection{Simulation studies}

\subsubsection{Logistic regression}

We fixed the sample size at $n=100$ and considered two different values of $p$, namely, $p=200$ and $p=400$.  The true coefficient vector $\theta^\star$ is set to have its first $s$ many components equal to 3 and the rest set to 0; here the cardinality $s$ can take two values, $s=4$ and $s=8$.  The rows of the design matrix $X$ are randomly simulated from a multivariate normal with mean 0, variance 1, and covariance matrix $\Sigma$, where $\Sigma_{ij}=r^{|i-j|}$ corresponds to a first-order autoregressive correlation structure. The correlation parameter $r$ takes values $r=0$ and $r=0.2$. The response variables $y_1,\ldots,y_n$ are generated independently, where $y_i$ has a Bernoulli distribution with success probability $\exp(x_i^\top\theta^\star)/\{1+\exp(x_i^\top\theta^\star)\}$, for $i=1,\ldots,n$.  The settings are determined by the triple $(p,s,r)$, so there are altogether eight simulation settings. There are 500 replications performed at each of the eight settings, and the variable selection results for the methods and metrics described above are summarized in Table~\ref{table:logistic}. 

\begin{table}[t]
\begin{center}  
\begin{tabular}{cccccccccccc}
    \hline
    $p$ & $|S|$ & $r$ & Metric & EB1 & EB2 & HS & lasso & alasso & SCAD & MCP & skinny\\
     \hline
         \multirow{3}{*}{200} &  \multirow{3}{*}{4} & \multirow{3}{*}{0} & TPR & 0.980 & 0.966 &	0.871& 0.998	&	0.701&	1.000	&0.999 & 0.997\\
         & & &  TNR &0.971& 0.992 &1.000&	0.953&	0.854 &0.962 &	0.986 &0.999 \\
          & & & MCC & 0.676	&0.859&	0.927	&0.630	&0.228	&0.597 &0.783 &0.980 \\[3pt]
             \multirow{3}{*}{200} &  \multirow{3}{*}{4} & \multirow{3}{*}{0.2}& TPR &  0.962 & 0.926 &	0.786&	0.995 &    0.806&	0.999 & 0.996 & 0.985 \\
         &  &  &TNR & 0.975& 0.994  &	1.000&	0.966 &0.847& 0.955& 0.983 & 0.999  \\
          &  & & MCC &  0.691 &0.856 &	0.878&	0.708&  0.272& 0.561& 0.744& 0.973 \\[3pt]
             \multirow{3}{*}{200} &  \multirow{3}{*}{8} & \multirow{3}{*}{0}& TPR &  0.746 & 0.643 &	0.276&	0.959 & 0.579&	0.962&	0.925 & 0.813 \\
         &  &  &TNR & 0.940	&0.983&	1.000&	0.895&	0.851&	0.948&	0.982 & 0.998 \\
          &  & & MCC &  0.519 & 0.618&	0.498&	0.528 &  0.206& 0.640& 0.790& 0.865 \\[3pt]
             \multirow{3}{*}{200} &  \multirow{3}{*}{8} & \multirow{3}{*}{0.2}& TPR &  0.747& 0.648 &0.287	&0.953&	0.705&	0.951&	0.888 &0.714 \\
         &  &  &TNR & 0.956&0.988&	1.000&	0.932& 0.837& 0.952& 0.984&0.998 \\
          &  & & MCC & 0.570& 0.663& 0.514& 0.631&  0.268&0.644 &0.780& 0.806 \\[3pt]
             \multirow{3}{*}{400} &  \multirow{3}{*}{4} & \multirow{3}{*}{0}& TPR & 0.922&	0.897& 0.582&	0.993& 0.741& 1.000&	1.000 &0.980\\
         &  &  &TNR & 0.989&	0.996&	1.000	&0.973&	0.924	&0.974&	0.991 & 0.998\\
          &  & & MCC &  0.728	&0.841&	0.728	&0.618&	0.265&0.543&	0.744 & 0.920 \\[3pt]
         \multirow{3}{*}{400} &  \multirow{3}{*}{4} & \multirow{3}{*}{0.2} & TPR & 0.926&0.860	&0.568	&0.994&	0.836&	0.998	&0.995 & 0.952\\
         & & &  TNR & 0.993	&0.998	&1.000&	0.980&	0.931	&0.971	&0.990 &0.998 \\
          & & & MCC & 0.795&	0.860&	0.737	&0.685	&0.336	&0.514	&0.718 & 0.899\\[3pt]
             \multirow{3}{*}{400} &  \multirow{3}{*}{8} & \multirow{3}{*}{0}& TPR &  0.490&	0.407&	0.057	&0.896&	0.468	&0.917	&0.843 & 0.498 \\
         &  &  &TNR & 0.974&	0.987&	1.000	&0.941	&0.930	&0.962	&0.987 & 0.996 \\
          &  & & MCC & 0.406	&0.401	&0.145&	0.493&	0.196&	0.546	&0.694 & 0.591 \\[3pt]
             \multirow{3}{*}{400} &  \multirow{3}{*}{8} & \multirow{3}{*}{0.2}& TPR &  0.548&	0.438&	0.092&	0.931	&0.617&	0.922	&0.842 & 0.565 \\
         &  &  &TNR & 0.983& 0.992 &1.000&	0.958	&0.921	&0.965&	0.989 &0.996 \\
          &  & & MCC &  0.513&	0.501	&0.226&	0.580	&0.260	&0.564&	0.714 &0.646 \\
    \hline
\end{tabular}
\end{center}
\caption{Comparison of TPR, TNR, and MCC for the two EB methods with different cutoffs and the six other methods across various settings in logistic regression.}
\label{table:logistic}
\end{table}

Our results showed that our EB method performs comparably to the other methods.  First, if the data analyst places higher priority on correctly identifying the active variables, then TPR would be his/her preferred metric.  In this case, EB1---with a smaller cutoff $t=0.1$, consistent with the higher priority on finding active variables---performs fairly well in terms of TPR compared to the other five methods. Lasso has the highest TPR in two of the eight settings but, as is common, it tends to over-select as evidenced by its low TNR.  SCAD has the highest TPR in the other six settingss. Horseshoe has the lowest TPR in all but one of the settings. This is because the standard/default strategy recommends using 95\% credible intervals, which is too conservative; a lower credibility level should be used if a less conservative selection procedure is desired.  Second, if the data analyst's priorities are more balanced, i.e., aiming for parsimonious models with good overall performance, then MCC would be the go-to metric and he/she would prefer the more balanced EB2 with a larger cutoff $t=0.5$.  Here, EB2 has an MCC comparable with the other methods. Adaptive lasso has the lowest MCC in six of the settings, whereas SkinnyGibbs has the highest across all eight settings.  This is not unexpected, given that skinnyGibbs is designed specifically for variable selection in logistic regression. 

In terms of runtime, we compared the methods by fixing $n=100$, $|S|=4$, and $r=0$, but varying $p \in \{200, 300, 400, 500, 600, 700, 800\}$. Each method was run 5 times to generate an average runtime. Figure~\ref{fig:logistictime} plots the runtime of our EB method and the skinnyGibbs method as $p$ increases.  This comparison might be rather striking, so some remarks are in order.  It is not possible for a Markov chain to explore the $S$-space thoroughly in an acceptable amount of time when $p$ is even moderately large.  Our proposed method is not designed to thoroughly explore the $S$-space; instead, the goal is to do a careful-but-admittedly-incomplete exploration focused on the high-probability configurations so that it can provide reliable variable selection. As our results show, apparently this latter goal can be accomplished competitively with far shorter runtimes.

\begin{figure}[t]
\begin{center}\scalebox{0.7}{\includegraphics{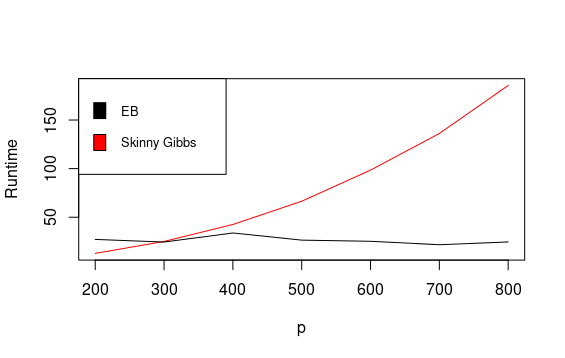}}
\end{center}
\caption{A comparison of the runtimes (in seconds) of EB and skinnyGibbs with fixed $n=100, |S|=4, r=0$ and varying $p=\{200, 300, 400, 500, 600, 700, 800\}$.}
\label{fig:logistictime}
\end{figure}

\subsubsection{Poisson regression}

The data $X$ is generated the same way as in the logistic regression settings, with one minor change---the common standard deviation across the rows of $X$ is set at $0.3$ instead of 1. This is to ensure that the Poisson variables generated are not exponentially large.  The response variables $y_1,\ldots,y_n$ are generated independently, with $y_i$ having a Poisson distribution with rate $\exp(x_i^\top\theta^\star)$. All other settings remain the same as in the logistic regression simulations above, with the same $(p,r,s)$ combinations. The Poisson regression simulation results for variable selection are summarized in Table~\ref{table:poisson},  where 100 replications were run at each simulation setting.

\begin{table}[t]
\begin{tabular}{@{}cccccccccc@{}}
\hline
    $p$ & $|S|$ &  $r$ & Metric & EB1 & EB2 & lasso & alasso & SCAD & MCP\\
\hline
         \multirow{3}{*}{200} &  \multirow{3}{*}{4} & \multirow{3}{*}{0} & TPR & 1.000	&1.000	&1.000&	0.998	&0.975&	0.950 \\
         & & &  TNR &0.999&	1.000&	0.893&	0.949&	0.983	&0.987 \\
          & & & MCC & 0.973&	0.998&	0.395	&0.547	&0.879	&0.906 \\[3pt]
             \multirow{3}{*}{200} &  \multirow{3}{*}{4} & \multirow{3}{*}{0.2}& TPR &  1.000&	1.000	&1.000	&1.000&	0.973	&0.903 \\
         &  &  &TNR & 1.000	&1.000&	0.907&	0.964&	0.995&	0.998 \\
          &  & & MCC &  0.992&	0.999&	0.425&	0.620&	0.905	&0.893 \\[3pt]
             \multirow{3}{*}{200} &  \multirow{3}{*}{8} & \multirow{3}{*}{0}& TPR &  1.000&	1.000	&0.979	&0.983&	0.598&	0.494 \\
         &  &  &TNR & 0.997&	1.000&	0.856&	0.915	&0.981	&0.989 \\
          &  & & MCC & 0.969	&0.998&	0.437	&0.555&	0.549	&0.514 \\[3pt]
             \multirow{3}{*}{200} &  \multirow{3}{*}{8} & \multirow{3}{*}{0.2}& TPR &  1.000	&1.000&0.983	&0.958	&0.514&	0.421 \\
         &  &  &TNR & 0.998	&1.000&	0.898&	0.946	&0.983&	0.988 \\
          &  & & MCC &  0.980&	0.998 &0.508&	0.625&	0.512	&0.470 \\[3pt]
             \multirow{3}{*}{400} &  \multirow{3}{*}{4} & \multirow{3}{*}{0}& TPR & 1.000	&1.000&	1.000&	1.000&	0.990&	0.975 \\
         &  &  &TNR & 1.000&	1.000&	0.926&	0.955	&0.995&	0.999\\
          &  & & MCC &  0.991&	1.000&	0.350	&0.445	&0.868&	0.936 \\[3pt]
         \multirow{3}{*}{400} &  \multirow{3}{*}{4} & \multirow{3}{*}{0.2} & TPR & 1.000&	1.000&	1.000&	0.998	&0.945	&0.863\\
         & & &  TNR &1.000&	1.000&	0.940&	0.971	&0.977&	0.978 \\
          & & & MCC & 0.994&	0.999&	0.392&	0.533&	0.863&	0.850\\[3pt]
             \multirow{3}{*}{400} &  \multirow{3}{*}{8} & \multirow{3}{*}{0}& TPR &  1.000	&1.000&0.979	&0.971&	0.494&	0.299 \\
         &  &  &TNR & 0.999	&1.000	&0.914&	0.942	&0.986&	0.992 \\
          &  & & MCC &0.985	&1.000		&0.417&	0.499&	0.434	&0.322 \\[3pt]
             \multirow{3}{*}{400} &  \multirow{3}{*}{8} & \multirow{3}{*}{0.2}& TPR &  1.000	&1.000&	0.960	&0.946&	0.415	&0.318 \\
         &  &  &TNR & 0.998&	1.000&	0.936	&0.960	&0.989	&0.993 \\
          &  & & MCC & 0.973&	0.999&	0.465	&0.552	&0.401	&0.357 \\
\hline
\end{tabular}
\caption{Comparison of TPR, TNR, and MCC for the two EB methods with different cutoffs and the four other methods across various settings in Poisson regression.}
\label{table:poisson}
\end{table}

We see that EB does very well compared to the other methods, and even better comparatively than in the logistic regression settings. EB2 has the highest MCC value in all eight configurations, and in fact, both EB1 and EB2 perform better than all other methods across all the settings for all three metrics. Comparing the methods, SCAD and MCP are not as competitive in Poisson regression as in logistic regression. 

For estimation, we use the same simulation settings as above, and compare the (empirical) mean squared error (MSE) values, $\E_{\theta^\star}\|\hat\theta - \theta^\star\|^2$, for the various methods. The results summarized in Table~\ref{table:est} show that our proposed method is an incredibly strong performer across all the settings. 

\begin{table}[t]
\begin{tabular}{@{}cccccccc@{}}
\hline
    $p$ & $|S|$ &  $r$ & EB & lasso & alasso & SCAD & MCP\\
\hline
         \multirow{1}{*}{200} &  \multirow{1}{*}{4} & \multirow{1}{*}{0} & 0.117	& 2.642	& 1.365& 2.064	&2.375 \\[3pt]
             \multirow{1}{*}{200} &  \multirow{1}{*}{4} & \multirow{1}{*}{0.2} &  0.094&	1.677	&0.813	&6.536& 7.714 \\[3pt]
             \multirow{1}{*}{200} &  \multirow{1}{*}{8} & \multirow{1}{*}{0} &  0.196& 10.143 & 9.906	&43.698& 48.180\\[3pt]
             \multirow{1}{*}{200} &  \multirow{1}{*}{8} & \multirow{1}{*}{0.2}&0.469&	16.201	&17.839	&56.941& 61.823 \\[3pt]
             \multirow{1}{*}{400} &  \multirow{1}{*}{4} & \multirow{1}{*}{0}&0.207&	3.925&	2.105&	2.184&	3.195 \\[3pt]
         \multirow{1}{*}{400} &  \multirow{1}{*}{4} & \multirow{1}{*}{0.2} & 0.100&	2.177&	1.176	&9.638	&10.754\\[3pt]
             \multirow{1}{*}{400} &  \multirow{1}{*}{8} & \multirow{1}{*}{0} & 4.691& 16.920& 14.248 & 59.989& 59.021 \\[3pt]
             \multirow{1}{*}{400} &  \multirow{1}{*}{8} & \multirow{1}{*}{0.2}&7.419& 17.451 &14.531& 60.560& 64.833 \\
\hline
\end{tabular}
\caption{Comparison of mean squared error (MSE) values for the EB method and the four other methods across various settings in Poisson regression.}
\label{table:est}
\end{table}


\section{Conclusion}
\label{S:discuss}

In this paper, we extend the empirical or data-driven prior specification strategy first proposed by \cite{martin.mess.walker.eb} to the case of high-dimensional GLMs and investigate its theoretical and practical performance.  In particular, we show that the proposed solution is ideal in the sense that it balances the strong theoretical performance that is necessary to justify its use in applications with the computational simplicity and efficiency necessary to be applicable in these problems.  The balance comes from the data-driven prior centering: we enjoy the computational advantage of a relatively simple, thin-tailed prior without subjecting ourselves to the theoretical sub-optimality that results from thin-tailed priors with fixed centers.  Compared to the linear models previously investigated, a challenge here in the GLM context is that there is no conjugacy in the prior and, therefore, no closed-form expressions for any of the posterior features.  These challenges affect both the theory and computation, but we have used some new techniques to successfully overcome them here in this paper.  While there are other methods available in the literature that have good empirical performance, and others that have powerful theoretical results, our contribution here is unique in the sense that our solution achieves both.  The solution is general and can be applied beyond logistic regression setting.  

Our numerical investigations in this paper focused on variable selection and point estimation, but the method itself is capable of answering other questions.  In a follow-up work it would be interesting to explore the performance of the proposed method in the context of uncertainty quantification about the coefficient vector $\theta$.  In the high-dimensional linear model setting, work has been done in \cite{martin.tang.jmlr} to look at how the method does in prediction---similar work with a focus on point prediction and uncertainty quantification of prediction can be carried out in this high-dimensional GLM context.  On the theoretical side, there are some new techniques employed here in the proofs, namely, relying on in-probability bounds as opposed to bounds in expectation.  The latter have their advantages, but the former are more flexible.  We expect that this added flexibility would be useful in other cases where, as is common, the priors would not be exactly conjugate.

\begin{appendix}
\section{Technical preliminaries}
\label{A:tech.prelim}

\subsection{Likelihood-related properties}
\label{AA:likelihood}

First, we present here two key results from \cite{narisetty2018skinny}, namely, Lemmas~A1 and A3, respectively, both concerning asymptotic properties of the likelihood function and MLEs in the specific case of high-dimensional logistic regression with Bernoulli response variables.  Lemmas~\ref{lem:Jn} and \ref{lem:mle} suitably generalize these two results. 

It is not enough for us to focus on the true $\theta^\star$ exclusively, so the results presented below cover general configurations $S$ that may or may not be supersets of $S(\theta^\star)$, i.e., where $\theta_S^\dagger$ is needed in place of $\theta_S^\star$.  Fortunately, the arguments \cite{narisetty2018skinny} used to prove their Lemmas~A1 and A3 go through almost word-for-word when replacing $\theta_S^\star$ with $\theta_S^\dagger$ where appropriate.  Finally, none of the considerations that follow make sense if $S = \varnothing$, e.g., if there is no parameter, then it does not make sense to ask about properties of the information matrix or about consistency of the MLE.  This detail is not relevant when only considering $S$ that are supersets of $S^\star$, but our analysis requires consideration of general $S$.  So, without loss of generality, wherever relevant, we restrict attention to $S \neq \varnothing$, so $|S| \geq 1$. 

The first result establishes an important continuity property for the observed information matrix. 

\begin{lem}
\label{lem:Jn}
Under Conditions~\ref{cond:dim}--\ref{cond:design}, for any fixed constant $c > 0$, there exists $\zeta_n \to 0$ such that 
\[ (1 - \zeta_n) \, J_n(S,\theta_S) \leq J_n(S,\theta_S^\dagger) \leq (1 + \zeta_n) \, J_n(S,\theta_S), \]
for any $(S,\theta_S)$ such that $|S| \, \|\theta_S - \theta_S^\dagger\|^2 = o(1)$.  
\end{lem}

Note, for a given $S$, the corresponding $\zeta_n = \zeta_{n,S}$ sequence is of the order 
\[ \zeta_n \sim \{ n^{-1} |S|^2 \Lambda_{|S|} \log p\}^{1/2}, \]
so the choice of $\zeta_n$ that holds uniformly over all the sufficiently low-complexity configurations is of the order $\{ n^{-1} s_n^2 \Lambda_{s_n} \log p\}^{1/2}$.  

The next result is a convergence rate for the MLE uniformly over configurations that are not ``too complex.'' This corresponds to Lemma~A3 in \cite{narisetty2018skinny} covering the case where the GLM's score function has sub-Gaussian tails.  The critical sub-exponential case is covered in \citet{lee.chae.glm}, based on key results in \citet{spokoiny2017}.

\begin{lem}
\label{lem:mle}
Under Conditions~\ref{cond:dim}--\ref{cond:design}, 
\[ \max_{S: |S|=s} \| \hat\theta_S - \theta_S^\dagger \|^2 = O_P(n^{-1} s \Lambda_s \log p), \quad n \to \infty, \]
uniformly over all $s$ with $1 \leq s \leq s_n$
\end{lem}


The following result is an important consequence of the two lemmas above.  The sub-Gaussian case is based on Lemma~3 in \cite{lee2021bayesian}, quoted as Equation~A7 of \cite{cao2020variable}, both based primarily on the analysis in \cite{narisetty2018skinny}.  The sub-exponential case is substantially more involved, and addressed in \citet{lee.chae.glm}.

\begin{lem}
\label{lem:loglik.diff}
Under Conditions~\ref{cond:dim}--\ref{cond:design}, there exists a constant $\kappa > 1$ such that 
\[ \ell_n(S, \hat\theta_S) - \ell_n(S^\star, \hat\theta_{S^\star}) \leq \kappa (\log p)(|S|-|S^\star|), \quad \text{with $\prob_{\theta^\star}$-probability $\to 1$}, \]
uniformly over $S$ with $S \supset S^\star$ and $|S| \leq s_n$.  
\end{lem}

\subsection{Marginal likelihood}
\label{AA:marginal}

Here we provide justification for the approximation \eqref{eq:S.post} of the marginal posterior $\pi^n$ at configuration $S$.  Recall that the marginal posterior satisfies 
\[ \pi^n(S) = \frac{\pi_n(S) \, m_n(S)}{\sum_{S'} \pi_n(S) \, m_n(S)}, \]
where $m_n(S)$ is the marginal likelihood for configuration $S$:
\[ m_n(S) = \int_{\RR^{|S|}} \underbrace{L_n(S,\theta_S)^\alpha \, \nm_{|S|}(\theta_S \mid \hat\theta_S, \gamma J_n(S,\hat\theta_S)^{-1})}_{\text{$= g(\theta_S)$, say}} \, d\theta_S. \]
So the goal is to lower- and upper-bound the integral $m_n(S)$. Aside from providing justification for our simple computational strategy, the result in Lemma~\ref{lem:laplace} below will be useful in the proofs of our main results below.  In particular, we will have a need to bound 
\[ \frac{\pi^n(S)}{\pi^n(S^\star)} = \frac{\pi_n(S) \, m_n(S)}{\pi_n(S^\star) \, m_n(S^\star)} \]

\begin{lem}
\label{lem:laplace}
Under Conditions~\ref{cond:dim}--\ref{cond:design}, the marginal likelihood $m_y(S)$ satisfies 
\[ 1 \leq \frac{m_n(S)}{(1 + \alpha \gamma)^{-|S|/2} \, L_n(S,\hat\theta_S)^\alpha} \leq 1 + e^{-C|S| \Lambda_{|S|} \log p}, \]
with $\prob_{\theta^\star}$-probability tending to 1, for a constant $C > 0$.  Consequently, 
\begin{equation}
\label{eq:prob.ratio}
\frac{\pi^n(S)}{\pi^n(S^\star)} \leq 2 \, \frac{\pi_n(S)}{\pi_n(S^\star)} (1 + \alpha \gamma)^{-(|S|-|S^\star|)/2} \, \frac{L_n(S,\hat\theta_S)^\alpha}{L_n(S^\star, \hat\theta_{S^\star})^\alpha}.
\end{equation}
\end{lem}

\begin{proof}
Thanks to Lemma~\ref{lem:mle}, it is safe to assume that the MLEs $\hat\theta_S$ are all within a small neighborhood of their respective targets $\theta_S^\dagger$.  Split the marginal likelihood integral into two parts according to $\RR^{|S|} = A_S \cup A_S^c$, where $A_S = \{\theta_S: \|\theta_S - \hat\theta_S\|^2 \leq r_n^2(S)\}$,  and $r_n^2(S) = n^{-1} |S| \Lambda_{|S|} \log p$. For $\theta_S \in A_S$, by Taylor's theorem and Lemma~\ref{lem:Jn}, we get 
\begin{align*}
\ell_n(S,\theta_S) - \ell_n(S,\hat\theta_S) & \geq -\tfrac12 (1+\zeta_n) (\theta_S - \hat\theta_S)^\top J_n(S,\hat\theta_S) (\theta_S - \hat\theta_S) \\
\ell_n(S,\theta_S) - \ell_n(S,\hat\theta_S) & \leq -\tfrac12 (1-\zeta_n) (\theta_S - \hat\theta_S)^\top J_n(S,\hat\theta_S) (\theta_S - \hat\theta_S). 
\end{align*}
The first of the above two displays gives a lower bound on the marginal likelihood,
\begin{align*}
m_n(S) & \geq \int_{A_S} g(\theta_S) \, d\theta_S \\
& \geq \{1 + \alpha \gamma / (1+\zeta_n)\}^{-|S|/2} \, L_n(S,\hat\theta_S)^\alpha \\
& \geq (1 + \alpha \gamma)^{-|S|/2} \Bigl\{ \frac{1 + \alpha \gamma/(1 + \zeta_n)}{1 + \alpha\gamma}\Bigr\}^{-|S|/2} \, L_n(S,\hat\theta_S)^\alpha \\
& \geq (1 + \alpha \gamma)^{-|S|/2} \, L_n(S,\hat\theta_S)^\alpha,
\end{align*}
which agrees with the familiar Laplace approximation expression used in \eqref{eq:S.post}.  Similarly, for an upper bound on the marginal likelihood, we get 
\begin{align*}
m_n(S) & = \Bigl( \int_{A_S} + \int_{A_S^c} \Bigr) g(\theta_S) \, d\theta_S \leq (1 + \alpha \gamma)^{-|S|/2} L_n(S,\hat\theta_S)^\alpha + \int_{A_S^c} g(\theta_S) \, d\theta_S,
\end{align*}
where, again, the first term in the above bound agrees with the Laplace approximation expression in \eqref{eq:S.post}. For $\theta_S \in A_S^c$, we have $\ell_n(S,\theta_S) - \ell_n(S,\hat\theta_S) \leq -Cnr_n^2(S)$, so 
\[ \int_{A_S^c} g(\theta_S) \, d\theta_S \leq L_n(S,\hat\theta_S)^\alpha e^{-\alpha C n\delta_n^2(S)} \, \Pi_{n,S}(A_S^c) \leq L_n(S,\hat\theta_S)^\alpha e^{-\alpha C nr_n^2(S)}. \]
The prior probability of $A_S^c$ has a non-trivial, exponentially small upper bound, but it is no smaller than the other exponentially small term in the above display, so its inclusion does not improve the overall bound.  Since the above arguments all hold with probability tending to 1, this completes the proof of the lemma's first claim.  For the second claim, we consider the ratio 
\[ \frac{m_n(S)}{m_n(S^\star)} \leq \frac{(1 + \alpha\gamma)^{-|S|/2} \{1 + e^{-C n r_n^2(S)}\}}{(1 + \alpha\gamma)^{-|S^\star|/2}} \, \frac{L_n(S,\hat\theta_S)^\alpha}{L_n(S^\star, \hat\theta_{S^\star})^\alpha}.  \]
Then the second claim follows since the term in curly braces above is bounded by 2, uniformly in $S$.  
\end{proof}

\subsection{Empirical priors and posterior concentration}
\label{AA:empirical.priors}

We briefly describe the general framework in \cite{martin2019data} for establishing posterior concentration rates with empirical priors in high-dimensional problems.  They put forward sufficient conditions on the empirical prior, one they called {\em local} and the other {\em global}.  In what follows, let $\Pi_n = (\pi_n, \Pi_{n,S})$ be a general empirical prior for $\Theta = (S,\Theta_S)$, and $\Pi^n = \Pi^{n,\alpha}$ the corresponding posterior that uses power $\alpha \in (0,1)$ on the likelihood function.  Also, let $\eps_n = \eps_n(\theta^\star)$ be a generic deterministic sequence that satisfies $\eps_n \to 0$ and $n \eps_n^2 \to \infty$ and may depend on the true $\theta^\star$.  

\begin{lp}
Given $\eps_n$, there exists constants $B,D > 0$ such that 
\begin{align}
\pi_n(S^\star) \gtrsim e^{-Bn\eps_n^2}, \quad \text{for all large $n$}, \label{eq:lp.S} \\
\prob_{\theta^\star}\bigl[ \Pi_{n,S^\star}\{{\cal L}_n(S^\star)\} > e^{-Dn\eps_n^2} \bigr] & \to 1, \quad n \to \infty, \label{eq:lp.theta}
\end{align}
where 
\[ {\cal L}_n(S^\star) = \{\theta \in \RR^p: L_n(S^\star,\theta_{S^\star}) \geq e^{-dn\eps_n^2} L_n(S^\star, \hat\theta_{S^\star})\}, \quad d>0, \]
\end{lp}

\begin{gp}
Given $\eps_n$, there exists $G > 0$ and $m > 1$ such that 
\begin{equation}
\label{eq:gp}
\sum_S \pi_n(S) \int [ \E_{\theta^\star}\{ \pi_{n,S}(\theta_S)^m \} ]^{1/m} \, d\theta_S \lesssim e^{G n\eps_n}, \quad \text{$n$ large}.
\end{equation}
\end{gp}

Under these conditions, a convergence rate in terms of Hellinger distance between joint distributions follows; see Theorem~2 in \cite{martin2019data}.  In regression cases like the GLMs under consideration here, the Hellinger rate in terms of joint distributions implies the same rate for the root average squared conditional Hellinger distances in \eqref{eq:hellinger}; see Appendix~\ref{proof:rate} below for details.

\section{Proofs}
\label{S:proofs}

\subsection{Proof of Theorem~\ref{thm:h.rate}}
\label{proof:rate}

The proof proceeds by first checking that the local and global prior conditions, as described above, are met under the conditions stated in Theorem~\ref{thm:h.rate} above.  Then Theorem~2 of \cite{martin2019data} implies the Hellinger rate result in \eqref{eq:h.rate}.  Since Theorem~\ref{thm:dim} holds independently and under the same conditions as the theorem we are currently proving, we can assume, where relevant, that $S$ is such that $|S| \leq C|S^\star|$.  

\begin{lem}
\label{lem:lp}
Under Conditions~\ref{cond:dim}--\ref{cond:prior}, our proposed empirical prior satisfies the local prior condition as described above, with $\eps_n(\theta^\star) = (n^{-1} |S(\theta^\star)| \log p)^{1/2}$. 
\end{lem}

\begin{proof}
Fix $\theta^\star$ and set $s^\star = |S(\theta^\star)|$. The first part of the local prior condition is easy to check, with $n\eps_n^2 = s^\star \log p$.  Indeed, using the inequality ${p \choose s} \leq p^s$ we get
\[ \pi_n(S^\star) = \frac{(p^{-a})^{s^\star}}{{p \choose s^\star}} \geq e^{-(1 + a) s^\star \log p} = e^{-(1+a) n\eps_n^2}, \]
so the bound in \eqref{eq:lp.S} holds with $B=1+a > 0$.  

Next, for \eqref{eq:lp.theta}, the data-dependent neighborhood ${\cal L}_n(S^\star)$ is given by 
\[ {\cal L}_n(S^\star) = \{\theta_{S^\star}: \ell_n(S^\star, \theta_{S^\star}) - \ell_n(S^\star, \hat\theta_{S^\star}) > -dn\eps_n^2\}. \]
Since $\ell_n$ is concave, this is a bounded neighborhood of $\hat\theta_{S^\star}$.  It is a relatively small neighborhood too, since the log-likelihood is of order $n$; this means that Lemma~\ref{lem:Jn} applies to any $\theta_{S^\star} \in {\cal L}_n(S^\star)$ and to $\hat\theta_{S^\star}$.  For $\theta_{S^\star} \in {\cal L}_n(S^\star)$, Taylor's theorem implies
\[ \ell_n(S^\star, \theta_{S^\star}) - \ell_n(S^\star, \hat\theta_{S^\star}) = -\tfrac12 (\theta_{S^\star} - \hat\theta_{S^\star})^\top J_n(S^\star, \tilde\theta_{S^\star}) (\theta_{S^\star} - \hat\theta_{S^\star}), \]
where $\tilde\theta_{S^\star} \in {\cal L}_n(S^\star)$ satisfies $\|\tilde\theta_{S^\star} - \hat\theta_{S^\star}\| \leq \|\theta_{S^\star} - \hat\theta_{S^\star}\|$.  Applying Lemma~\ref{lem:Jn} first with $\tilde\theta_{S^\star}$ and then with $\hat\theta_{S^\star}$ gives 
\begin{equation}
\label{eq:J.bound}
J_n(n,\tilde\theta_{S^\star}) \leq c_n \, J_n(S^\star, \hat\theta_{S^\star}), \quad \text{with probability $\to 1$}, 
\end{equation}
where $c_n = (1 - \zeta_n)(1 + \zeta_n) \to 1$.  Therefore, 
\[ {\cal L}_n(S^\star) \supset \{ \theta_{S^\star}: \gamma^{-1} (\theta_{S^\star} - \hat\theta_{S^\star})^\top J_n(S^\star, \hat\theta_{S^\star}) (\theta_{S^\star} - \hat\theta_{S^\star}) < \delta_n \}, \]
where $\delta_n := 2d n\eps_n^2/c_n \gamma \gg p^{-1}$.  The prior probability of the event on the right-hand side of the above display is the probability that $Z < \delta_n$, where $Z \sim {\sf ChiSq}(s^\star)$.  Therefore, 
\begin{align*}
\Pi_{n,S^\star}\{{\cal L}_n(S^\star)\} & = \prob( Z < \delta_n ) \\
& = \frac{1}{2^{s^\star/2} \Gamma(s^\star/2)} \int_0^{\delta_n} z^{s^\star/2-1} e^{-z/2} \, dz \\
& \geq \frac{e^{-\delta_n/2}}{2^{s^\star/2} \Gamma(s^\star/2)} \int_0^{\delta_n} z^{s^\star/2 - 1} \, dz \\
& = \frac{2e^{-\delta_n/2} \delta_n^{s^\star / 2}}{s^\star 2^{s^\star/2} \Gamma(s^\star/2)}.
\end{align*}
The lower bound on $\Pi_n\{{\cal L}_n(S^\star)\}$ is $ \geq e^{-D n \eps_n^2}=e^{-D s^\star \log p}$ for some $D > 0$, as was to be shown.  The prior probability bound holds surely, so \eqref{eq:lp.theta} follows from the ``probability $\to 1$'' conclusion in \eqref{eq:J.bound}.  
\end{proof}

\begin{lem}
\label{lem:gp}
Under Conditions~\ref{cond:dim}--\ref{cond:prior}, our proposed empirical prior satisfies the global prior condition as described above, with $\eps_n(\theta^\star) = (n^{-1} |S(\theta^\star)| \log p)^{1/2}$. 
\end{lem}

\begin{proof}
The empirical prior density $\pi_{n,S}$ is given by 
\[ \pi_{n,S}(\theta_S) = (2\pi)^{-|S|/2} |\gamma^{-1} J_n(S,\hat\theta_S)|^{1/2} \exp\bigl\{ -\tfrac{1}{2\gamma} (\theta_S - \hat\theta_S)^\top J_n(S,\hat\theta_S) (\theta_S - \hat\theta_S) \bigr\}. \]
Lemma~\ref{lem:mle} establishes the MLE bounds
\begin{equation}
\label{eq:mle.event}
\|\hat\theta_S - \theta_S^\dagger\|^2 \lesssim n^{-1} |S| \Lambda_{|S|} \log p, \quad \text{uniformly in $S$ with $|S| \lesssim |S^\star|$}, 
\end{equation}
and, therefore, with probability tending to 1,
\[ (1 - \zeta_n) J_n(S,\theta_S^\dagger) \leq J_n(S,\hat\theta_S) \leq (1 + \zeta_n) J_n(S,\theta_S^\dagger). \]
Let ${\cal E}_n$ denote the event in \eqref{eq:mle.event}; since $\prob_{\theta^\star}({\cal E}_n) \to 1$, we will restrict attention to cases where ${\cal E}_n$ holds in what follows.  Then 
\begin{align*}
\pi_{n,S}(\theta_S) & \leq (2\pi)^{-|S|/2} (1 + \zeta_n)^{|S|/2} |\gamma^{-1} J_n(S,\theta_S^\dagger)|^{1/2} \\
& \qquad \times \exp\bigl\{ -\tfrac{1 - \zeta_n}{2\gamma} (\theta_S - \hat\theta_S)^\top J_n(S,\theta_S^\dagger) (\theta_S - \hat\theta_S) \bigr\}.
\end{align*}
If $\Sigma_S = \gamma (1-\zeta_n)^{-1} J_n(S,\theta_S^\dagger)^{-1}$, then the upper bound can be simplified as 
\[ \pi_{n,S}(\theta_S) \leq \bigl( \tfrac{1+\zeta_n}{1-\zeta_n} \bigr)^{|S|/2} \, (2\pi)^{-|S|/2} |\Sigma_S|^{-1/2} \exp\{-\tfrac12 (\theta_S - \hat\theta_S)^\top \Sigma_S^{-1} (\theta_S - \hat\theta_S)\}. \]
Write $\theta_S - \hat\theta_S = (\theta_S - \theta_S^\dagger) + (\theta_S^\dagger - \hat\theta_S)$ and then expand the quadratic form in the above display to get 
\[ (\theta_S - \hat\theta_S)^\top \Sigma_S^{-1} (\theta_S - \hat\theta_S) \geq (\theta_S - \theta_S^\dagger)^\top \Sigma_S^{-1} (\theta_S - \theta_S^\dagger) + 2(\theta_S - \theta_S^\dagger)^\top \Sigma_S^{-1} (\hat\theta_S - \theta_S^\dagger). \]
Then it is easy to check that 
\[ \pi_{n,S}(\theta_S) \leq \bigl( \tfrac{1+\zeta_n}{1-\zeta_n} \bigr)^{|S|/2} e^{|(\theta_S - \theta_S^\dagger)^\top \Sigma_S^{-1} (\hat\theta_S - \theta_S^\dagger)|} \, \nm_{|S|}(\theta_S \mid \theta_S^\dagger, \Sigma_S). \]
Apply Cauchy--Schwarz to the quadratic form in the exponent above gives 
\[ |(\theta_S - \theta_S^\dagger)^\top \Sigma_S^{-1} (\hat\theta_S - \theta_S^\dagger)| \leq \|\Sigma_S^{-1/2}(\theta_S - \theta_S^\dagger)\|  \, \| \Sigma_S^{-1/2} (\hat\theta_S - \theta_S^\dagger)\|. \]
By Condition~\ref{cond:design} and Lemma~\ref{lem:mle}, the second term above can be bounded as 
\begin{align*}
\| \Sigma_S^{-1/2} (\hat\theta_S - \theta_S^\dagger)\|^2 & \leq \gamma^{-1} n \Lambda_{|S|} \, \|\hat\theta_S - \theta_S^\dagger\|^2 \lesssim \gamma^{-1} \,|S| \, \Lambda_{|S|}^2 \, \log p. 
\end{align*}
By Condition~\ref{cond:prior}(a), this is upper bounded by a constant times $|S|\log p$.  Let $t_S^2 \sim |S|\log p$ denote that upper bound.  Then the empirical prior density is bounded as
\[ \pi_{n,S}(\theta_S) \leq \bigl( \tfrac{1+\zeta_n}{1-\zeta_n} \bigr)^{|S|/2} e^{t_S \|\Sigma_S^{-1/2}(\theta_S - \theta_S^\dagger)\|} \, \nm_{|S|}(\theta_S \mid \theta_S^\dagger, \Sigma_S). \]
This is constant in data $y$, so the expectation in \eqref{eq:gp} can be ignored---and $m$ can be arbitrarily close to 1.  Moreover, the integral in \eqref{eq:gp} over $\theta_S$ can now be upper bounded by the moment generating function of a chi distribution, with $|S|$ degrees of freedom, evaluated at $t_S$.  This moment generating function does not have a convenient closed-form expression---it involves the confluent hypergeometric function---but since $t_S$ is large (proportional to $\log p$) for all $S$, we can apply the standard asymptotic approximation of the chi distribution's moment generating function \citep[][Ch.~13]{abramowitz.stegun.1966} to get 
\[ \int e^{t_S \|\Sigma_S^{-1/2}(\theta_S - \theta_S^\dagger)\|} \, \nm_{|S|}(\theta_S \mid \theta_S^\dagger, \Sigma_S) \, d\theta_S \lesssim e^{G |S| \log p}, \]
for some constant $G > 0$.  Multiplying by $\{(1+\zeta_n)/(1-\zeta_n)\}^{|S|/2}$ does not affect the bound.  Averaging the bound $e^{G |S| \log p}$ over low-complexity $S$'s, with $|S| \lesssim s^\star$, is upper bounded by $e^{G s^\star \log p} = e^{G n \eps_n^2}$, which proves \eqref{eq:gp}. 
\end{proof}

\subsection{Proof of Theorem~\ref{thm:no.super.sets}} 
\label{proof:no.super.sets}

By \eqref{eq:prob.ratio}, the marginal posterior mass function $\pi^n$ satisfies 
\[ \frac{\pi^n(S)}{\pi^n(S^\star)} \lesssim \frac{\pi_n(S)}{\pi_n(S^\star)} (1 + \alpha \gamma)^{-(|S|-|S^\star|)/2} \, \exp[\alpha\{ \ell_n(S,\hat\theta_S) - \ell_n(S^\star, \hat\theta_{S^\star})\}], \]
where the constant baked into ``$\lesssim$'' is 2.  By Lemma~\ref{lem:loglik.diff}, with probability converging to 1, the exponential term is uniformly upper-bounded in $S$ with $S \supset S^\star$ and $|S| \leq s_n$ by $\exp\{\alpha \kappa (\log p) (|S| - |S^\star|)\}$. Then the prior mass ratio satisfies 
\[ \frac{\pi_n(S)}{\pi_n(S^\star)} \lesssim \frac{{p \choose |S^\star|}}{{p \choose |S|}} \, p^{-a(|S|-|S^\star|)}, \]
so summing the (limiting) upper bound over all $S \supset S^\star$ gives 
\begin{align*}
\sum_{S: S \supset S^\star} \pi^n(S) & \lesssim \sum_{S: S \supset S^\star} \frac{\pi^n(S)}{\pi^n(S^\star)} \\
& = \sum_{s = |S^\star|+1}^{s_n} \frac{{p \choose |S^\star|} {p - |S^\star| \choose p-s}}{{p \choose s}} \{(1 + \alpha \gamma)^{-1/2}\}^{s-|S^\star|} p^{-(\beta - \alpha \kappa) (s - |S^\star|)} \\
& \leq \sum_{s = |S^\star|+1}^{s_n} s^{s-|S^\star|} p^{-(\beta - \alpha \kappa) (s - |S^\star|)} \\
& \leq \sum_{s = |S^\star|+1}^{s_n} (s_n p^{-(\beta - \alpha \kappa)})^{s - |S^\star|}. 
\end{align*}
Since $s_n < p^{\beta-\alpha\kappa}$ by Condition~\ref{cond:prior}, the dominating series converges and, therefore, the tail of that same series must form a divergent sequence as $|S^\star| \to \infty$, which proves the claim.

\subsection{Proof of Theorem~\ref{thm:dim}}

Those $S$ with $|S| > C |S^\star|$ that are proper supersets of $S^\star$ have already been covered in the proof of Theorem~\ref{thm:no.super.sets} above.  So it suffices to consider $S$ that are large but exclude some important variables.  Define the mapping $S \to S^+ = S \cup S^\star$.  The only part of the posterior $\pi^n$ that depends on $S$ itself---not just on $|S|$---is the likelihood component, and the likelihood is increasing in complexity, i.e.,  
\[ L_n(S,\hat\theta_S) \leq L_n(S^+, \hat\theta_{S^+}). \]
So, if $\S = \{S: C|S^\star| < |S| \leq s_n \text{ and } S \not\supset S^\star\}$, then we can proceed as follows:
\begin{align*}
\sum_{S \in \S} \pi^n(S) & \lesssim \sum_{S \in \S} \frac{\pi^n(S)}{\pi^n(S^\star)} \\
& \leq \sum_{S \in \S} \frac{\pi_n(S)}{\pi_n(S^\star)} (1 + \alpha \gamma)^{-(|S|-|S^\star|)/2} \, e^{\alpha\{ \ell_n(S^+,\hat\theta_{S^+}) - \ell_n(S^\star, \hat\theta_{S^\star})\}} \\
& \leq p^{\alpha \kappa |S^\star|} \sum_{s > C|S^\star|}  s_n^{s-|S^\star|} p^{-(\beta-\alpha K)(s - |S^\star|)} \\
& = p^{\alpha \kappa C|S^\star|} (s_n p^{-a})^{(C-1)|S^\star|} \times O(1).
\end{align*}
Since $s_n \ll p^\beta$ and $\beta > \alpha \kappa C / (C - 1)$ by definition of $C$, the upper bound vanishes, proving the claim.

\subsection{Proof of Theorem~\ref{thm:consistent}}

Take any fixed $\theta^\star$ that meets the stated conditions, and set $S^\star = S(\theta^\star)$ and $s^\star = |S^\star|$.  Define $S \in \S := \{S: |S| \leq C s^\star \text{ and } S \not\supseteq S^\star\}$, where $C > 1$ is as in the statement of Theorem~\ref{thm:dim}.  The key point is that there is at least one important variable omitted in the models $S \in \S$.  Let $\rho_n^2 = \rho_n^2(\theta^\star) = n^{-1} v \log p$ denote the lower bound on $\theta_j^{\star 2}$ for $j \in S^\star$, as defined in \eqref{eq:beta.min}, where $v = v(s^\star) = c s^\star \Lambda_{cs^\star}$.   In their Supplementary Materials, \cite{narisetty2018skinny} showed that, with probability tending to 1,\footnote{The result that they {\em stated}, i.e., $\ell_n(S,\hat\theta_S) - \ell_n(S^\star, \hat\theta_{S^\star}) \lesssim -n \rho_n^2$, is incorrect---the difference should depend on how close $S$ is to $S^\star$.  But the result that they {\em proved} is the one stated here.} 
\[ \ell_n(S,\hat\theta_S) - \ell_n(S^\star, \hat\theta_{S^\star}) \leq  -F |S \setminus S^\star| n \rho_n^2, \quad \text{uniformly over $S \in \S$}, \]
for a constant $F > 0$.  Then we get the bound 
\begin{align*}
\sum_{S \in \S} \pi^n(S) & \lesssim \sum_{S \in \S} \frac{\pi^n(S)}{\pi^n(S^\star)} \\
& \leq \sum_{S \in \S} \frac{\pi_n(S)}{\pi_n(S^\star)} (1 + \alpha \gamma)^{-(|S|-|S^\star|)/2} \, e^{\alpha\{ \ell_n(S,\hat\theta_S) - \ell_n(S^\star, \hat\theta_{S^\star})\}} \\
& \leq \sum_{S \in \S} \frac{{p\choose s^\star}}{{p\choose |S|}}  (1+\alpha\gamma)^{-(|S|-s^\star)/2} p^{-\beta(|S|-s^\star) - \alpha F v (|S|-s^\star)},
\end{align*}
where we used the fact that $|S \setminus S^\star| \geq |S| - s^\star$.  Since the summands only depend on $|S|$, the sum over $S \in \S$ can be simplified by first choosing the overall size of the model, then choosing the size of $S \cap S^\star$.  That is, 
\[ \sum_{S \in \S} (\cdots) = \sum_{s=0}^{Cs^\star} \sum_{t=0}^{s \wedge (s^\star-1)} {s^\star\choose t} {p-s^\star\choose s-t} \, (\cdots), \]
where $t$ indexes the size of $S \cap S^\star$.  Note that $t$ can be at most $s^\star-1$ since $S$ is not allowed to be a superset of $S^\star$. Plugging in the expression for $(\cdots)$ and using the bound 
\[ \frac{{s^\star \choose t} {p-s^\star\choose s-t}{p\choose s^\star}}{{p\choose s}} \leq s^{s-t} p^{s^\star - t}, \]
we get 
\[ \sum_{S \in \S} \pi^n(S) \leq \sum_{s=0}^{Cs^\star} \sum_{t=0}^{s \wedge (s^\star-1)} (\phi p^{-\beta})^{s - s^\star}  s^{s-t} (p^{1 - \alpha J v})^{s^\star - t}, \]
where $\phi = (1 + \alpha\gamma)^{-1/2}$. Split the outer sum on the right-hand side above into two pieces: $s \leq s^\star - 1$ and $s \geq s^\star$.  For the first sum, 
\begin{align*}
\sum_{s=0}^{s^\star-1} \sum_{t=0}^s (\phi p^{-\beta})^{s - s^\star} s^{s-t} (p^{1 - \alpha F v})^{s^\star - t} & = \sum_{s=0}^{s^\star-1} \Bigl( \frac{p^\beta}{s\phi} \Bigr)^{s^\star-s} \sum_{t=0}^s (s p^{1-\alpha F v})^{s^\star-t} \\
& \lesssim \sum_{s=0}^{s^\star - 1} (\phi^{-1} p^{1 + \beta - \alpha F v})^{s^\star-s} \\ 
& \lesssim \phi^{-1} p^{1 + \beta - \alpha F v}.
\end{align*}
We have that $\alpha F v > 1 + \beta$ because $v=v(s^\star)$ is or can be made large: if $s^\star \to \infty$ then $v \to \infty$ or, otherwise, the constant $c > 1$ baked into $v$ can be chosen sufficiently large.  Since $\phi^{-1}$ is linear in $\gamma$, which is at most polynomial in $n$, the negative power of $p$ dominates so the bound is $o(1)$. Similarly, for the second sum
\begin{align*}
\sum_{s=s^\star}^{Cs^\star} \sum_{t=0}^{s^\star-1} (\phi p^{-\beta})^{s - s^\star} s^{s-t} (p^{1 - \alpha F v})^{s^\star - t} & = \sum_{s=s^\star}^{Cs^\star} \Bigl( \frac{p^\beta}{s\phi} \Bigr)^{s^\star-s} \sum_{t=0}^{s^\star-1} (s p^{1-\alpha F v})^{s^\star-t} \\
& \lesssim s^\star p^{1 - \alpha F v} \sum_{s=s^\star}^{Cs^\star} \Bigl( \frac{s^\star \phi}{p^\beta} \Bigr)^{s-s^\star} \\
& \lesssim s^\star p^{1 - \alpha F v},
\end{align*}
where the last ``$\lesssim$'' follows because $p^\beta \gg s^\star$ and $\phi < 1$.  By the same argument as given for the first summation, the remaining term is $o(1)$, so we can conclude that $\sum_{S \in \S} \pi^n(S) \to 0$ with probability tending to 1, which proves the claim.

\subsection{Proof of Theorem~\ref{thm:bvm}}

To start, write the posterior distribution $\Pi^n$ as 
\[ \Pi^n(d\theta) = \sum_S \pi^n(S) \, \Pi_S^n(d\theta_S) \otimes \delta_{0_{S^c}}(d\theta_{S^c}), \]
where $\Pi_S^n$ is the conditional posterior of $\theta_S$, given $S$, having a density with respect to Lebesgue measure on the corresponding $|S|$-dimensional Euclidean space.  Then the total variation distance between $\Pi^n$ and the Gaussian approximation $\Psi^n$ in \eqref{eq:gaussian.limit} is 
\begin{align*}
d_\text{\sc tv}\bigl( \Pi^n, \Psi^n \bigr) & \leq \sum_S \pi^n(S) \, d_\text{\sc tv}\bigl( \Pi_S^n \otimes \delta_{0_{S^c}}, \, \Psi^n \bigr) \\
& = \pi^n(S^\star) \, d_\text{\sc tv}\bigl( \Pi_{S^\star}^n \otimes \delta_{0_{S^{\star c}}}, \, \Psi^n \bigr) + \sum_{S \neq S^\star} \pi^n(S) \, d_\text{\sc tv}\bigl( \Pi_S^n \otimes \delta_{0_{S^c}}, \, \Psi^n \bigr).
\end{align*}
Since the total variation distance is bounded by 2 and $\sum_{S \neq S^\star} \pi^n(S) \to 0$ in $\prob_{\theta^\star}$-probability by Theorem~\ref{thm:consistent}, the latter term is $o_{\prob_{\theta^\star}}(1)$.  It remains to show the same for $d_\text{\sc tv}\bigl( \Pi_{S^\star}^n \otimes \delta_{0_{S^{\star c}}}, \, \Psi^n \bigr)$.  Since both distributions concentrate on the same $S^\star$ configuration, we can drop the ``$\otimes$'' terms both above and in \eqref{eq:gaussian.limit}.  That is, the goal is simply to bound 
\[ d_\text{\sc tv}\bigl\{ \Pi_{S^\star}^n , \,  \nm_{|S^\star|}\bigl( \hat\theta_{S^\star}, \varrho \, J_n(S^\star, \hat\theta_{S^\star})^{-1} \bigr)\bigr\}. \]
The posterior $\Pi_{S^\star}^n$ has a density $\pi_{S^\star}^n$ with respect to Lebesgue measure; similarly, the normal distribution in the above display has a density with respect to Lebesgue measure, which we will denote by $\psi_{S^\star}^n$.  Following the argument used to establish the Laplace approximation result in Lemma~\ref{lem:laplace}, it suffices to focus our attention here on integrals over $\theta_S$ in a sufficiently small neighborhood of the MLE $\hat\theta_S$. Once localized, using the form of the proposed empirical prior, we can apply the continuity property in Lemma~\ref{lem:Jn} to get (locally) pointwise lower and upper bounds, respectively, on the density $\pi_{S^\star}^n$, i.e., 
\[ \underbrace{\frac{L_n^\alpha(S^\star, \hat\theta_{S^\star})}{m_n(S^\star) \, \{1 + \alpha\gamma(1 \pm \zeta_n)\}^{|S^\star|/2}}}_{=G_n^\pm} \, \underbrace{\nm\bigl(\theta_{S^\star} \mid \hat\theta_{S^\star}, \, \tfrac{\gamma}{1 + \alpha\gamma(1 \pm \zeta_n)} J_n(S^\star, \hat\theta_{S^\star})^{-1} \bigr)}_{= g_n^\pm(\theta_{S^\star})},  \]
where $m_n(S^\star)$ denotes the marginal likelihood under configuration $S^\star$ and $\zeta_n = \zeta_{n,S^\star}$ is the vanishing sequence identified in Lemma~\ref{lem:Jn}; more on $\zeta_n$ below.  Denote the lower and upper bounds as $\underline\pi_{S^\star}^n$ and $\overline\pi_{S^\star}^n$, respectively.  If we can show that both the lower and upper bounds have vanishing $L_1$-distance to the Gaussian approximation, then we are done. And since both of the bounds have the same form, we will focus here on the upper bound $\overline\pi_{S^\star}^n$, the one where ``$\pm = -$.'' A simple triangle inequality-type argument gives the bound 
\[ | \overline\pi_{S^\star}^n(\cdot) - \psi_{S^\star}^n(\cdot) | \leq G_n^- | g_n^-(\cdot) - \psi_{S^\star}^n(\cdot) | + |G_n^- - 1| \psi_{S^\star}^n(\cdot), \]
and, consequently, 
\[ d_\text{\sc tv}( \overline\pi_{S^\star}^n, \psi_{S^\star}^n ) \leq G_n^- \, d_\text{\sc tv}(g_n^-, \psi_{S^\star}^n) + |G_n^- - 1|. \]
It follows from Lemma~\ref{lem:laplace} that the $G_n^- \to 1$ in $\prob_{\theta^\star}$-probability as $n \to \infty$.  So it remains to bound the total variation distance in the above display, which is between two Gaussians with a common mean and very similar variances. But a special case of the general result in Theorem~1.8 of \citet{arbas.etal.2023}---see, also, \citet{devroye.etal.2023}---gives that 
\[ d_\text{\sc tv}(g_n^-, \psi_{S^\star}^n) \lesssim \Bigl| \varrho \div \frac{\gamma}{1 + \alpha\gamma(1 - \zeta_n)} - 1 \Bigr| \, |S^\star|^{1/2}. \]
It is easy to check that the upper bound is of the order $\zeta_n |S^\star|^{1/2}$, so we need to consider how quickly $\zeta_n$ is vanishing.  Recall that, following the statement of Lemma~\ref{lem:Jn}, it was noted that, for a particular configuration, in this case $S=S^\star$, the choice of $\zeta_n$ was of the order $\{ n^{-1} |S^\star|^2 \Lambda_{|S^\star|} \log p\}^{1/2}$.  That means 
\[ d_\text{\sc tv}(g_n^-, \psi_{S^\star}^n) \lesssim \{ n^{-1} |S^\star|^3 \Lambda_{|S^\star|} \log p \}^{1/2}. \]
By the additional configuration size condition \eqref{eq:small.model}, the upper bound above is vanishing.  This proves that $d_\text{\sc tv}\bigl( \Pi^n, \Psi^n \bigr) \to 0$ in $\prob_{\theta^\star}$-probability and, since it is also uniformly bounded, the theorem's statement (in term of expectations) follows from the dominated convergence theorem. 

\end{appendix}

\begin{acks}[Acknowledgments]
The authors thank the Editor and anonymous Associate Editor and reviewers for their helpful feedback on a previous version of the manuscript. 
\end{acks}

\begin{funding}
This work was partially supported by the U.~S.~National Science Foundation, under grants DMS--1737933, DMS--1811802, and SES--205122. 
\end{funding}

\bibliographystyle{imsart-nameyear} 
\bibliography{mybib.bib}       

\begin{thebibliography}{54}

\bibitem[\protect\citeauthoryear{Abramovich and
  Grinshtein}{2010}]{abramovich.grinshtein.2010}
\begin{barticle}[author]
\bauthor{\bsnm{Abramovich},~\bfnm{Felix}\binits{F.}} \AND
  \bauthor{\bsnm{Grinshtein},~\bfnm{Vadim}\binits{V.}}
(\byear{2010}).
\btitle{M{AP} model selection in {G}aussian regression}.
\bjournal{Electronic Journal of Statistics}
\bvolume{4}
\bpages{932--949}.
\bmrnumber{2721039}
\end{barticle}
\endbibitem

\bibitem[\protect\citeauthoryear{Abramovich and
  Grinshtein}{2016}]{abramovich2016model}
\begin{barticle}[author]
\bauthor{\bsnm{Abramovich},~\bfnm{Felix}\binits{F.}} \AND
  \bauthor{\bsnm{Grinshtein},~\bfnm{Vadim}\binits{V.}}
(\byear{2016}).
\btitle{Model selection and minimax estimation in generalized linear models}.
\bjournal{IEEE Transactions on Information Theory}
\bvolume{62}
\bpages{3721--3730}.
\end{barticle}
\endbibitem

\bibitem[\protect\citeauthoryear{Abramowitz and
  Stegun}{1966}]{abramowitz.stegun.1966}
\begin{bbook}[author]
\bauthor{\bsnm{Abramowitz},~\bfnm{Milton}\binits{M.}} \AND
  \bauthor{\bsnm{Stegun},~\bfnm{Irene~A.}\binits{I.~A.}}
(\byear{1966}).
\btitle{Handbook of {M}athematical {F}unctions}.
\bpublisher{Dover}, \baddress{New York}.
\bmrnumber{0208797}
\end{bbook}
\endbibitem

\bibitem[\protect\citeauthoryear{Arbas, Ashtiani and
  Liaw}{2023}]{arbas.etal.2023}
\begin{binproceedings}[author]
\bauthor{\bsnm{Arbas},~\bfnm{Jamil}\binits{J.}},
  \bauthor{\bsnm{Ashtiani},~\bfnm{Hassan}\binits{H.}} \AND
  \bauthor{\bsnm{Liaw},~\bfnm{Christopher}\binits{C.}}
(\byear{2023}).
\btitle{Polynomial time and private learning of unbounded {G}aussian mixture
  models}.
In \bbooktitle{Proceedings of the 40th International Conference on Machine
  Learning}.
\bseries{ICML'23}.
\bpublisher{JMLR.org}.
\end{binproceedings}
\endbibitem

\bibitem[\protect\citeauthoryear{Arias-Castro and
  Lounici}{2014}]{ariascastro.lounici.2014}
\begin{barticle}[author]
\bauthor{\bsnm{Arias-Castro},~\bfnm{Ery}\binits{E.}} \AND
  \bauthor{\bsnm{Lounici},~\bfnm{Karim}\binits{K.}}
(\byear{2014}).
\btitle{Estimation and variable selection with exponential weights}.
\bjournal{Electronic Journal of Statistics}
\bvolume{8}
\bpages{328--354}.
\bmrnumber{3195119}
\end{barticle}
\endbibitem

\bibitem[\protect\citeauthoryear{Barber, Drton and
  Tan}{2016}]{barber2016laplace}
\begin{bincollection}[author]
\bauthor{\bsnm{Barber},~\bfnm{Rina~Foygel}\binits{R.~F.}},
  \bauthor{\bsnm{Drton},~\bfnm{Mathias}\binits{M.}} \AND
  \bauthor{\bsnm{Tan},~\bfnm{Kean~Ming}\binits{K.~M.}}
(\byear{2016}).
\btitle{Laplace approximation in high-dimensional {B}ayesian regression}.
In \bbooktitle{Statistical Analysis for High-Dimensional Data}
\bpages{15--36}.
\bpublisher{Springer}.
\end{bincollection}
\endbibitem

\bibitem[\protect\citeauthoryear{Barbieri and
  Berger}{2004}]{barbieri.berger.2004}
\begin{barticle}[author]
\bauthor{\bsnm{Barbieri},~\bfnm{Maria~Maddalena}\binits{M.~M.}} \AND
  \bauthor{\bsnm{Berger},~\bfnm{James~O.}\binits{J.~O.}}
(\byear{2004}).
\btitle{Optimal predictive model selection}.
\bjournal{Ann. Statist.}
\bvolume{32}
\bpages{870--897}.
\bmrnumber{2065192}
\end{barticle}
\endbibitem

\bibitem[\protect\citeauthoryear{Belitser and
  Ghosal}{2020}]{belitser2020empirical}
\begin{barticle}[author]
\bauthor{\bsnm{Belitser},~\bfnm{Eduard}\binits{E.}} \AND
  \bauthor{\bsnm{Ghosal},~\bfnm{Subhashis}\binits{S.}}
(\byear{2020}).
\btitle{Empirical {B}ayes oracle uncertainty quantification for regression}.
\bjournal{The Annals of Statistics}
\bvolume{48}
\bpages{3113--3137}.
\end{barticle}
\endbibitem

\bibitem[\protect\citeauthoryear{Belitser and
  Nurushev}{}]{belitser.nurushev.uq}
\begin{barticle}[author]
\bauthor{\bsnm{Belitser},~\bfnm{Eduard}\binits{E.}} \AND
  \bauthor{\bsnm{Nurushev},~\bfnm{Nurzhan}\binits{N.}}
\btitle{{Needles and straw in a haystack: Robust confidence for possibly sparse
  sequences}}.
\bjournal{Bernoulli}
\bvolume{26}
\bpages{191--225}.
\end{barticle}
\endbibitem

\bibitem[\protect\citeauthoryear{Bhadra et~al.}{2017}]{bhadra.hsplus.2017}
\begin{barticle}[author]
\bauthor{\bsnm{Bhadra},~\bfnm{Anindya}\binits{A.}},
  \bauthor{\bsnm{Datta},~\bfnm{Jyotishka}\binits{J.}},
  \bauthor{\bsnm{Polson},~\bfnm{Nicholas~G.}\binits{N.~G.}} \AND
  \bauthor{\bsnm{Willard},~\bfnm{Brandon}\binits{B.}}
(\byear{2017}).
\btitle{The horseshoe+ estimator of ultra-sparse signals}.
\bjournal{Bayesian Analysis}
\bvolume{12}
\bpages{1105--1131}.
\bmrnumber{3724980}
\end{barticle}
\endbibitem

\bibitem[\protect\citeauthoryear{Bhadra et~al.}{2019a}]{bhadra2019lasso}
\begin{barticle}[author]
\bauthor{\bsnm{Bhadra},~\bfnm{Anindya}\binits{A.}},
  \bauthor{\bsnm{Datta},~\bfnm{Jyotishka}\binits{J.}},
  \bauthor{\bsnm{Polson},~\bfnm{Nicholas~G}\binits{N.~G.}} \AND
  \bauthor{\bsnm{Willard},~\bfnm{Brandon}\binits{B.}}
(\byear{2019}a).
\btitle{Lasso meets horseshoe: A survey}.
\bjournal{Statistical Science}
\bvolume{34}
\bpages{405--427}.
\end{barticle}
\endbibitem

\bibitem[\protect\citeauthoryear{Bhadra et~al.}{2019b}]{bhadra2016}
\begin{barticle}[author]
\bauthor{\bsnm{Bhadra},~\bfnm{Anindya}\binits{A.}},
  \bauthor{\bsnm{Datta},~\bfnm{Jyotishka}\binits{J.}},
  \bauthor{\bsnm{Li},~\bfnm{Yunfan}\binits{Y.}},
  \bauthor{\bsnm{Polson},~\bfnm{Nicholas~G.}\binits{N.~G.}} \AND
  \bauthor{\bsnm{Willard},~\bfnm{Brandon}\binits{B.}}
(\byear{2019}b).
\btitle{Prediction risk for the horseshoe regression}.
\bjournal{Journal of Machine Learning Research (JMLR)}
\bvolume{20}
\bpages{Paper No. 78, 39}.
\bmrnumber{3960932}
\end{barticle}
\endbibitem

\bibitem[\protect\citeauthoryear{Bondell and Reich}{2012}]{bondell.reich.2012}
\begin{barticle}[author]
\bauthor{\bsnm{Bondell},~\bfnm{Howard~D.}\binits{H.~D.}} \AND
  \bauthor{\bsnm{Reich},~\bfnm{Brian~J.}\binits{B.~J.}}
(\byear{2012}).
\btitle{Consistent high-dimensional {B}ayesian variable selection via penalized
  credible regions}.
\bjournal{Journal of the American Statistical Association}
\bvolume{107}
\bpages{1610--1624}.
\bmrnumber{3036420}
\end{barticle}
\endbibitem

\bibitem[\protect\citeauthoryear{B{\"u}hlmann}{2011}]{buhlmann2011comment}
\begin{barticle}[author]
\bauthor{\bsnm{B{\"u}hlmann},~\bfnm{Peter}\binits{P.}}
(\byear{2011}).
\btitle{Comments on `{R}egression shrinkage and selection via the lasso: {A}
  retrospective'}.
\bjournal{Journal of the Royal Statistical Society: Series B (Statistical
  Methodology)}
\bvolume{73}
\bpages{277-279}.
\end{barticle}
\endbibitem

\bibitem[\protect\citeauthoryear{B{\"u}hlmann and van~de
  Geer}{2011}]{buhlmann.geer.book}
\begin{bbook}[author]
\bauthor{\bsnm{B{\"u}hlmann},~\bfnm{Peter}\binits{P.}} \AND
  \bauthor{\bparticle{van~de} \bsnm{Geer},~\bfnm{Sara}\binits{S.}}
(\byear{2011}).
\btitle{Statistics for High-Dimensional Data}.
\bseries{Springer Series in Statistics}.
\bpublisher{Springer, Heidelberg}.
\bmrnumber{2807761}
\end{bbook}
\endbibitem

\bibitem[\protect\citeauthoryear{Cao and Lee}{2020}]{cao2020variable}
\begin{barticle}[author]
\bauthor{\bsnm{Cao},~\bfnm{Xuan}\binits{X.}} \AND
  \bauthor{\bsnm{Lee},~\bfnm{Kyoungjae}\binits{K.}}
(\byear{2020}).
\btitle{Variable Selection Using Nonlocal Priors in High-Dimensional
  Generalized Linear Models With Application to {fMRI} Data Analysis}.
\bjournal{Entropy}
\bvolume{22}
\bpages{807}.
\end{barticle}
\endbibitem

\bibitem[\protect\citeauthoryear{Carvalho, Polson and
  Scott}{2010}]{carvalho.polson.scott.2010}
\begin{barticle}[author]
\bauthor{\bsnm{Carvalho},~\bfnm{Carlos~M.}\binits{C.~M.}},
  \bauthor{\bsnm{Polson},~\bfnm{Nicholas~G.}\binits{N.~G.}} \AND
  \bauthor{\bsnm{Scott},~\bfnm{James~G.}\binits{J.~G.}}
(\byear{2010}).
\btitle{The horseshoe estimator for sparse signals}.
\bjournal{Biometrika}
\bvolume{97}
\bpages{465--480}.
\bmrnumber{2650751}
\end{barticle}
\endbibitem

\bibitem[\protect\citeauthoryear{Castillo, Schmidt-Hieber and van~der
  Vaart}{2015}]{castillo.schmidt.vaart.reg}
\begin{barticle}[author]
\bauthor{\bsnm{Castillo},~\bfnm{Isma\"el}\binits{I.}},
  \bauthor{\bsnm{Schmidt-Hieber},~\bfnm{Johannes}\binits{J.}} \AND
  \bauthor{\bparticle{van~der} \bsnm{Vaart},~\bfnm{Aad}\binits{A.}}
(\byear{2015}).
\btitle{Bayesian linear regression with sparse priors}.
\bjournal{The Annals of Statistics}
\bvolume{43}
\bpages{1986--2018}.
\bmrnumber{3375874}
\end{barticle}
\endbibitem

\bibitem[\protect\citeauthoryear{Castillo and van~der
  Vaart}{2012}]{castillo.vaart.2012}
\begin{barticle}[author]
\bauthor{\bsnm{Castillo},~\bfnm{Isma{\"e}l}\binits{I.}} \AND
  \bauthor{\bparticle{van~der} \bsnm{Vaart},~\bfnm{Aad}\binits{A.}}
(\byear{2012}).
\btitle{{Needles and straw in a haystack: Posterior concentration for possibly
  sparse sequences}}.
\bjournal{The Annals of Statistics}
\bvolume{40}
\bpages{2069--2101}.
\end{barticle}
\endbibitem

\bibitem[\protect\citeauthoryear{Chicco and Jurman}{2023}]{chicco2023matthews}
\begin{barticle}[author]
\bauthor{\bsnm{Chicco},~\bfnm{Davide}\binits{D.}} \AND
  \bauthor{\bsnm{Jurman},~\bfnm{Giuseppe}\binits{G.}}
(\byear{2023}).
\btitle{The Matthews correlation coefficient (MCC) should replace the ROC AUC
  as the standard metric for assessing binary classification}.
\bjournal{BioData Mining}
\bvolume{16}
\bpages{4}.
\end{barticle}
\endbibitem

\bibitem[\protect\citeauthoryear{Devroye, Mehrabian and
  Reddad}{2023}]{devroye.etal.2023}
\begin{bunpublished}[author]
\bauthor{\bsnm{Devroye},~\bfnm{L.}\binits{L.}},
  \bauthor{\bsnm{Mehrabian},~\bfnm{A.}\binits{A.}} \AND
  \bauthor{\bsnm{Reddad},~\bfnm{T.}\binits{T.}}
(\byear{2023}).
\btitle{The total variation distance between high-dimensional {G}aussians with
  the same mean}.
\bnote{{\tt arXiv:1810.08693}}.
\end{bunpublished}
\endbibitem

\bibitem[\protect\citeauthoryear{Fang and Ghosh}{2023}]{fang2023high}
\begin{bunpublished}[author]
\bauthor{\bsnm{Fang},~\bfnm{Xiao}\binits{X.}} \AND
  \bauthor{\bsnm{Ghosh},~\bfnm{Malay}\binits{M.}}
(\byear{2023}).
\btitle{High-dimensional properties for empirical priors in linear regression
  with unknown error variance}.
\bnote{{\em Statistical Papers}, to appear}.
\end{bunpublished}
\endbibitem

\bibitem[\protect\citeauthoryear{Friedman, Hastie and
  Tibshirani}{2010}]{glmnet}
\begin{barticle}[author]
\bauthor{\bsnm{Friedman},~\bfnm{Jerome}\binits{J.}},
  \bauthor{\bsnm{Hastie},~\bfnm{Trevor}\binits{T.}} \AND
  \bauthor{\bsnm{Tibshirani},~\bfnm{Robert}\binits{R.}}
(\byear{2010}).
\btitle{Regularization Paths for Generalized Linear Models via Coordinate
  Descent}.
\bjournal{Journal of Statistical Software}
\bvolume{33}
\bpages{1--22}.
\end{barticle}
\endbibitem

\bibitem[\protect\citeauthoryear{George and
  McCullogh}{1993}]{george.mccullogh.1993}
\begin{barticle}[author]
\bauthor{\bsnm{George},~\bfnm{Edward~I.}\binits{E.~I.}} \AND
  \bauthor{\bsnm{McCullogh},~\bfnm{Robert~E.}\binits{R.~E.}}
(\byear{1993}).
\btitle{Variable selection via Gibbs sampling}.
\bjournal{Journal of the American Statistical Association}
\bvolume{88}
\bpages{881--889}.
\end{barticle}
\endbibitem

\bibitem[\protect\citeauthoryear{Goodrich et~al.}{2022}]{goodrich2018rstanarm}
\begin{bmisc}[author]
\bauthor{\bsnm{Goodrich},~\bfnm{Ben}\binits{B.}},
  \bauthor{\bsnm{Gabry},~\bfnm{Jonah}\binits{J.}},
  \bauthor{\bsnm{Ali},~\bfnm{Imad}\binits{I.}} \AND
  \bauthor{\bsnm{Brilleman},~\bfnm{Sam}\binits{S.}}
(\byear{2022}).
\btitle{rstanarm: {Bayesian} applied regression modeling via {Stan}.}
\bnote{R package version 2.21.3}.
\end{bmisc}
\endbibitem

\bibitem[\protect\citeauthoryear{Gr{\"u}nwald and
  Mehta}{2020}]{grunwald.mehta.rates}
\begin{barticle}[author]
\bauthor{\bsnm{Gr{\"u}nwald},~\bfnm{Peter~D}\binits{P.~D.}} \AND
  \bauthor{\bsnm{Mehta},~\bfnm{Nishant~A}\binits{N.~A.}}
(\byear{2020}).
\btitle{Fast rates for general unbounded loss functions: from ERM to
  generalized Bayes}.
\bjournal{The Journal of Machine Learning Research}
\bvolume{21}
\bpages{2040--2119}.
\end{barticle}
\endbibitem

\bibitem[\protect\citeauthoryear{Gr\"unwald and van
  Ommen}{2017}]{grunwald.ommen.scaling}
\begin{barticle}[author]
\bauthor{\bsnm{Gr\"unwald},~\bfnm{Peter}\binits{P.}} \AND
  \bauthor{\bparticle{van} \bsnm{Ommen},~\bfnm{Thijs}\binits{T.}}
(\byear{2017}).
\btitle{Inconsistency of {B}ayesian inference for misspecified linear models,
  and a proposal for repairing it}.
\bjournal{Bayesian Analysis}
\bvolume{12}
\bpages{1069--1103}.
\bmrnumber{3724979}
\end{barticle}
\endbibitem

\bibitem[\protect\citeauthoryear{Hastie, Tibshirani and
  Friedman}{2009}]{hastie.tibshirani.friedman.2009}
\begin{bbook}[author]
\bauthor{\bsnm{Hastie},~\bfnm{Trevor}\binits{T.}},
  \bauthor{\bsnm{Tibshirani},~\bfnm{Robert}\binits{R.}} \AND
  \bauthor{\bsnm{Friedman},~\bfnm{Jerome}\binits{J.}}
(\byear{2009}).
\btitle{The Elements of Statistical Learning},
\bedition{2nd} ed.
\bpublisher{Springer-Verlag}, \baddress{New York}.
\bmrnumber{1851606}
\end{bbook}
\endbibitem

\bibitem[\protect\citeauthoryear{Jeong and Ghosal}{2021}]{jeong2021posterior}
\begin{barticle}[author]
\bauthor{\bsnm{Jeong},~\bfnm{Seonghyun}\binits{S.}} \AND
  \bauthor{\bsnm{Ghosal},~\bfnm{Subhashis}\binits{S.}}
(\byear{2021}).
\btitle{Posterior contraction in sparse generalized linear models}.
\bjournal{Biometrika}
\bvolume{108}
\bpages{367--379}.
\end{barticle}
\endbibitem

\bibitem[\protect\citeauthoryear{Lee and Cao}{2021}]{lee2021bayesian}
\begin{barticle}[author]
\bauthor{\bsnm{Lee},~\bfnm{Kyoungjae}\binits{K.}} \AND
  \bauthor{\bsnm{Cao},~\bfnm{Xuan}\binits{X.}}
(\byear{2021}).
\btitle{Bayesian group selection in logistic regression with application to MRI
  data analysis}.
\bjournal{Biometrics}
\bvolume{77}
\bpages{391--400}.
\end{barticle}
\endbibitem

\bibitem[\protect\citeauthoryear{Lee and Chae}{2024}]{lee.chae.glm}
\begin{bunpublished}[author]
\bauthor{\bsnm{Lee},~\bfnm{Jeyong}\binits{J.}} \AND
  \bauthor{\bsnm{Chae},~\bfnm{Minwoo}\binits{M.}}
(\byear{2024}).
\btitle{On {B}ayesian model selection consistency for high-dimensional
  generalized linear models}.
\bnote{In preparation}.
\end{bunpublished}
\endbibitem

\bibitem[\protect\citeauthoryear{Liu and Martin}{2019}]{liu2019empirical}
\begin{bunpublished}[author]
\bauthor{\bsnm{Liu},~\bfnm{Chang}\binits{C.}} \AND
  \bauthor{\bsnm{Martin},~\bfnm{Ryan}\binits{R.}}
(\byear{2019}).
\btitle{An empirical ${G}$-{W}ishart prior for sparse high-dimensional
  {G}aussian graphical models}.
\bnote{{\tt arXiv:1912.03807}}.
\end{bunpublished}
\endbibitem

\bibitem[\protect\citeauthoryear{Liu, Martin and Shen}{2023}]{ebpiece}
\begin{bunpublished}[author]
\bauthor{\bsnm{Liu},~\bfnm{C.}\binits{C.}},
  \bauthor{\bsnm{Martin},~\bfnm{R.}\binits{R.}} \AND
  \bauthor{\bsnm{Shen},~\bfnm{W.}\binits{W.}}
(\byear{2023}).
\btitle{Empirical priors and posterior concentration in a piecewise polynomial
  sequence model}.
\bnote{{\em Statistica Sincica}, to appear; {\tt arXiv:1712.03848}}.
\end{bunpublished}
\endbibitem

\bibitem[\protect\citeauthoryear{Martin}{2019}]{ebmono}
\begin{barticle}[author]
\bauthor{\bsnm{Martin},~\bfnm{Ryan}\binits{R.}}
(\byear{2019}).
\btitle{Empirical priors and posterior concentration rates for a monotone
  density}.
\bjournal{Sankhy\=a A.}
\bvolume{81}
\bpages{493--509}.
\end{barticle}
\endbibitem

\bibitem[\protect\citeauthoryear{Martin, Mess and
  Walker}{2017}]{martin.mess.walker.eb}
\begin{barticle}[author]
\bauthor{\bsnm{Martin},~\bfnm{Ryan}\binits{R.}},
  \bauthor{\bsnm{Mess},~\bfnm{Raymond}\binits{R.}} \AND
  \bauthor{\bsnm{Walker},~\bfnm{Stephen~G.}\binits{S.~G.}}
(\byear{2017}).
\btitle{Empirical {B}ayes posterior concentration in sparse high-dimensional
  linear models}.
\bjournal{Bernoulli}
\bvolume{23}
\bpages{1822--1847}.
\bmrnumber{3624879}
\end{barticle}
\endbibitem

\bibitem[\protect\citeauthoryear{Martin and Ning}{2020}]{martin2020empirical}
\begin{barticle}[author]
\bauthor{\bsnm{Martin},~\bfnm{Ryan}\binits{R.}} \AND
  \bauthor{\bsnm{Ning},~\bfnm{Bo}\binits{B.}}
(\byear{2020}).
\btitle{Empirical priors and coverage of posterior credible sets in a sparse
  normal mean model}.
\bjournal{Sankhya A}
\bvolume{82}
\bpages{477--498}.
\end{barticle}
\endbibitem

\bibitem[\protect\citeauthoryear{Martin and Tang}{2020}]{martin.tang.jmlr}
\begin{barticle}[author]
\bauthor{\bsnm{Martin},~\bfnm{Ryan}\binits{R.}} \AND
  \bauthor{\bsnm{Tang},~\bfnm{Yiqi}\binits{Y.}}
(\byear{2020}).
\btitle{Empirical Priors for Prediction in Sparse High-dimensional Linear
  Regression}.
\bjournal{Journal of Machine Learning Research}
\bvolume{21}
\bpages{1-30}.
\end{barticle}
\endbibitem

\bibitem[\protect\citeauthoryear{Martin and Walker}{2014}]{martin.walker.eb}
\begin{barticle}[author]
\bauthor{\bsnm{Martin},~\bfnm{Ryan}\binits{R.}} \AND
  \bauthor{\bsnm{Walker},~\bfnm{Stephen~G.}\binits{S.~G.}}
(\byear{2014}).
\btitle{Asymptotically minimax empirical {B}ayes estimation of a sparse normal
  mean vector}.
\bjournal{Electronic Journal of Statistics}
\bvolume{8}
\bpages{2188--2206}.
\end{barticle}
\endbibitem

\bibitem[\protect\citeauthoryear{Martin and Walker}{2019}]{martin2019data}
\begin{barticle}[author]
\bauthor{\bsnm{Martin},~\bfnm{Ryan}\binits{R.}} \AND
  \bauthor{\bsnm{Walker},~\bfnm{Stephen~G.}\binits{S.~G.}}
(\byear{2019}).
\btitle{Data-driven priors and their posterior concentration rates}.
\bjournal{Electronic Journal of Statistics}
\bvolume{13}
\bpages{3049--3081}.
\end{barticle}
\endbibitem

\bibitem[\protect\citeauthoryear{McCullagh and
  Nelder}{1989}]{McCullaghNelder:1989}
\begin{bbook}[author]
\bauthor{\bsnm{McCullagh},~\bfnm{Peter~M.}\binits{P.~M.}} \AND
  \bauthor{\bsnm{Nelder},~\bfnm{J.~A.}\binits{J.~A.}}
(\byear{1989}).
\btitle{{Generalized Linear Models}}.
\bpublisher{Chapman and Hall, London}.
\end{bbook}
\endbibitem

\bibitem[\protect\citeauthoryear{Miller and Dunson}{2019}]{miller.dunson.power}
\begin{barticle}[author]
\bauthor{\bsnm{Miller},~\bfnm{Jeffrey~W.}\binits{J.~W.}} \AND
  \bauthor{\bsnm{Dunson},~\bfnm{David~B.}\binits{D.~B.}}
(\byear{2019}).
\btitle{Robust {B}ayesian Inference via Coarsening}.
\bjournal{Journal of the American Statistical Association}
\bvolume{114}
\bpages{1113-1125}.
\end{barticle}
\endbibitem

\bibitem[\protect\citeauthoryear{Narisetty, Shen and
  He}{2019}]{narisetty2018skinny}
\begin{barticle}[author]
\bauthor{\bsnm{Narisetty},~\bfnm{Naveen~N.}\binits{N.~N.}},
  \bauthor{\bsnm{Shen},~\bfnm{Juan}\binits{J.}} \AND
  \bauthor{\bsnm{He},~\bfnm{Xuming}\binits{X.}}
(\byear{2019}).
\btitle{Skinny {G}ibbs: a consistent and scalable {G}ibbs sampler for model
  selection}.
\bjournal{Journal of the American Statistical Association}
\bvolume{114}
\bpages{1205--1217}.
\bmrnumber{4011773}
\end{barticle}
\endbibitem

\bibitem[\protect\citeauthoryear{Piironen and
  Vehtari}{2017}]{piironen2017sparsity}
\begin{barticle}[author]
\bauthor{\bsnm{Piironen},~\bfnm{Juho}\binits{J.}} \AND
  \bauthor{\bsnm{Vehtari},~\bfnm{Aki}\binits{A.}}
(\byear{2017}).
\btitle{Sparsity information and regularization in the horseshoe and other
  shrinkage priors}.
\bjournal{Electronic Journal of Statistics}
\bvolume{11}
\bpages{5018--5051}.
\end{barticle}
\endbibitem

\bibitem[\protect\citeauthoryear{Rigollet}{2012}]{rigollet2012}
\begin{barticle}[author]
\bauthor{\bsnm{Rigollet},~\bfnm{Philippe}\binits{P.}}
(\byear{2012}).
\btitle{Kullback-{L}eibler aggregation and misspecified generalized linear
  models}.
\bjournal{The Annals of Statistics}
\bvolume{40}
\bpages{639--665}.
\bmrnumber{2933661}
\end{barticle}
\endbibitem

\bibitem[\protect\citeauthoryear{Rigollet and
  Tsybakov}{2012}]{rigollet.tsybakov.2012}
\begin{barticle}[author]
\bauthor{\bsnm{Rigollet},~\bfnm{Philippe}\binits{P.}} \AND
  \bauthor{\bsnm{Tsybakov},~\bfnm{Alexandre~B.}\binits{A.~B.}}
(\byear{2012}).
\btitle{Sparse estimation by exponential weighting}.
\bjournal{Statist. Sci.}
\bvolume{27}
\bpages{558--575}.
\bmrnumber{3025134}
\end{barticle}
\endbibitem

\bibitem[\protect\citeauthoryear{Shun and McCullagh}{1995}]{shun1995laplace}
\begin{barticle}[author]
\bauthor{\bsnm{Shun},~\bfnm{Zhenming}\binits{Z.}} \AND
  \bauthor{\bsnm{McCullagh},~\bfnm{Peter}\binits{P.}}
(\byear{1995}).
\btitle{Laplace approximation of high dimensional integrals}.
\bjournal{Journal of the Royal Statistical Society: Series B (Methodological)}
\bvolume{57}
\bpages{749--760}.
\end{barticle}
\endbibitem

\bibitem[\protect\citeauthoryear{Spokoiny}{2017}]{spokoiny2017}
\begin{barticle}[author]
\bauthor{\bsnm{Spokoiny},~\bfnm{Vladimir}\binits{V.}}
(\byear{2017}).
\btitle{Penalized maximum likelihood estimation and effective dimension}.
\bjournal{Annales de l'Institut Henri Poincar\'{e} Probabilit\'{e}s et
  Statistiques}
\bvolume{53}
\bpages{389--429}.
\bmrnumber{3606746}
\end{barticle}
\endbibitem

\bibitem[\protect\citeauthoryear{Syring and
  Martin}{2018}]{syring.martin.scaling}
\begin{barticle}[author]
\bauthor{\bsnm{Syring},~\bfnm{Nicholas}\binits{N.}} \AND
  \bauthor{\bsnm{Martin},~\bfnm{Ryan}\binits{R.}}
(\byear{2018}).
\btitle{Calibrating general posterior credible regions}.
\bjournal{Biometrika}
\bvolume{106}
\bpages{479--486}.
\end{barticle}
\endbibitem

\bibitem[\protect\citeauthoryear{Syring and
  Martin}{2023}]{syring.martin.subexp}
\begin{barticle}[author]
\bauthor{\bsnm{Syring},~\bfnm{Nicholas}\binits{N.}} \AND
  \bauthor{\bsnm{Martin},~\bfnm{Ryan}\binits{R.}}
(\byear{2023}).
\btitle{Gibbs posterior concentration rates under sub-exponential type losses}.
\bjournal{Bernoulli}
\bvolume{29}
\bpages{1080--1108}.
\end{barticle}
\endbibitem

\bibitem[\protect\citeauthoryear{Tang and Martin}{2021}]{tang2019package}
\begin{bmanual}[author]
\bauthor{\bsnm{Tang},~\bfnm{Yiqi}\binits{Y.}} \AND
  \bauthor{\bsnm{Martin},~\bfnm{Ryan}\binits{R.}}
(\byear{2021}).
\btitle{ebreg: Implementation of the empirical {B}ayes method}
\bnote{R package version 0.1.3}.
\end{bmanual}
\endbibitem

\bibitem[\protect\citeauthoryear{van~der Pas, Szab{\'o} and van~der
  Vaart}{2017}]{van2017uncertainty}
\begin{barticle}[author]
\bauthor{\bparticle{van~der} \bsnm{Pas},~\bfnm{St{\'e}phanie}\binits{S.}},
  \bauthor{\bsnm{Szab{\'o}},~\bfnm{Botond}\binits{B.}} \AND
  \bauthor{\bparticle{van~der} \bsnm{Vaart},~\bfnm{Aad}\binits{A.}}
(\byear{2017}).
\btitle{Uncertainty Quantification for the Horseshoe (with Discussion)}.
\bjournal{Bayesian Analysis}
\bvolume{12}
\bpages{1221--1274}.
\end{barticle}
\endbibitem

\bibitem[\protect\citeauthoryear{Walker and Hjort}{2001}]{walker.hjort.2001}
\begin{barticle}[author]
\bauthor{\bsnm{Walker},~\bfnm{Stephen}\binits{S.}} \AND
  \bauthor{\bsnm{Hjort},~\bfnm{Nils~Lid}\binits{N.~L.}}
(\byear{2001}).
\btitle{On {B}ayesian consistency}.
\bjournal{Journal of the Royal Statistical Society. Series B. Statistical
  Methodology}
\bvolume{63}
\bpages{811--821}.
\bmrnumber{1872068}
\end{barticle}
\endbibitem

\bibitem[\protect\citeauthoryear{Walker, Lijoi and
  Pr{\"u}nster}{2005}]{walker.lijoi.prunster.2005a}
\begin{barticle}[author]
\bauthor{\bsnm{Walker},~\bfnm{Stephen~G.}\binits{S.~G.}},
  \bauthor{\bsnm{Lijoi},~\bfnm{Antonio}\binits{A.}} \AND
  \bauthor{\bsnm{Pr{\"u}nster},~\bfnm{Igor}\binits{I.}}
(\byear{2005}).
\btitle{Data tracking and the understanding of {B}ayesian consistency}.
\bjournal{Biometrika}
\bvolume{92}
\bpages{765--778}.
\bmrnumber{2234184}
\end{barticle}
\endbibitem

\bibitem[\protect\citeauthoryear{Wei and Ghosal}{2020}]{wei2020contraction}
\begin{barticle}[author]
\bauthor{\bsnm{Wei},~\bfnm{Ran}\binits{R.}} \AND
  \bauthor{\bsnm{Ghosal},~\bfnm{Subhashis}\binits{S.}}
(\byear{2020}).
\btitle{Contraction properties of shrinkage priors in logistic regression}.
\bjournal{Journal of Statistical Planning and Inference}
\bvolume{207}
\bpages{215--229}.
\end{barticle}
\endbibitem

\end{thebibliography}

\end{document}